\documentclass[letterpaper,11pt]{amsart}
\usepackage{amsmath,amssymb,amsthm,pinlabel,tikz,hyperref,mathrsfs,thmtools}
\usepackage{verbatim}
\usepackage{float}

\usetikzlibrary{patterns}

\usepackage{xcolor}

\DeclareMathAlphabet{\mathpzc}{OT1}{pzc}{m}{it}

\newcommand{\nc}{\newcommand}
\nc{\dmo}{\DeclareMathOperator}
\dmo{\ra}{\rightarrow}
\dmo{\N}{\mathbb{N}}
\dmo{\Z}{\mathbb{Z}}
\dmo{\Q}{\mathbb{Q}}
\dmo{\R}{\mathbb{R}}
\dmo{\C}{\mathcal{C}}
\dmo{\AC}{\mathcal{AC}}
\dmo{\Mod}{Mod}
\dmo{\Sp}{Sp}
\dmo{\PMod}{PMod}
\dmo{\B}{B}
\dmo{\PB}{PB}
\dmo{\PR}{PSL(2,\mathbb{R})}
\dmo{\I}{\mathcal{I}}
\dmo{\el}{\ell_{\C}}
\dmo{\NN}{\mathcal{N}}
\dmo{\rk}{rank}
\dmo{\w}{\omega}
\dmo{\F}{\mathcal{F}}
\dmo{\Out}{\mathrm{Out}}
\dmo{\Hom}{\mathrm{Hom}}
\dmo{\sign}{sign}

\newcommand{\T}{\mathcal{T}}
\newcommand{\D}{\mathcal{D}}
\renewcommand{\H}{\mathcal{H}}

\usepackage{tikz-cd}

\tikzcdset{row sep/normal=1.5cm}
\tikzcdset{column sep/normal=1.5cm}

\usepackage{caption}
\usepackage{subcaption}

\usetikzlibrary{decorations.markings}
\tikzset{->-/.style={decoration={
  markings,
  mark=at position #1 with {\arrow{>}}},postaction={decorate}}}

\nc{\nt}{\newtheorem}

\nt{theorem}{Theorem}

\newtheorem{thm}{{\bf Theorem}}[section]
\newtheorem{lem}[thm]{{\bf Lemma}}
\newtheorem{cor}[thm]{{\bf Corollary}}
\newtheorem{prop}[thm]{{\bf Proposition}}

\newtheorem{remark}[thm]{Remark}

\newtheorem{definition}[thm]{Definition}
\numberwithin{equation}{section}

\title[Asymptotic translation length on the sphere complexes]{On the asymptotic translation lengths on the sphere complexes and the generalized fibered cone}

\date{\today}
\author{Hyungryul Baik}
\address{
		Department of Mathematical Sciences, KAIST\\
		291 Daehak-ro Yuseong-gu, Daejeon, 34141, South Korea 
}
\email{%
        hrbaik@kaist.ac.kr
}

\author{Dongryul M. Kim}

\address{%
	Department of Mathematics, Yale University\\
	219 Prospect Street, New Haven, CT 06511, USA
}
\email{%
	dongryul.kim@yale.edu
}

\author{Chenxi Wu}

\address{%
		Department of Mathematics, University of Wisconsin--Madison\\
		480 Lincoln Drive, Madison, WI 53706, USA
}
\email{%
        cwu367@math.wisc.edu
}

\begin{document}
\begin{abstract}

  In this paper, we study the asymptotic translation lengths on the sphere complexes of monodromies of a manifold fibered over the circle. Given a compact mapping torus, we define a cone in the first cohomology which we call the generalized fibered cone, and show that every primitive integral element gives a fibration over the circle. Moreover, we prove that the generalized fibered cone is a rational slice of Fried's cone, which is defined as the dual of homological directions, an analogue of Thurston's fibered cone.

  As a consequence of our description of the generalized fibered cone, we provide each proper subcone of the generalized fibered cone with a uniform upper bound for asymptotic translation lengths of monodromies on sphere complexes of fibers in the proper subcone. Our upper bound is purely in terms of the dimension of the proper subcone.
  We also deduce similar estimates for asymptotic translation lengths of some mapping classes on finite graphs constructed in the works of Dowdall--Kapovich--Leininger, measured on associated free-splitting complexes and free-factor complexes.

  Moreover, as an application of our result, we prove that the asymptote for the minimal asymptotic translation length of the genus $g$ handlebody group on the disk complex is $1/g^2$, the same as the one on the curve complex.
\end{abstract}

\maketitle

\tableofcontents



\section{Introduction} \label{sec:intro}

Group actions have been proven to be fruitful in the study of groups. For instance, Thurston \cite{thurston1988geometry} and Bers \cite{bers1978extremal} classified mapping classes of a closed surface according to the dynamics of their action on the Teichm\"uller space. Moreover, Masur--Minsky \cite{MasurMinsky99, masur2000geometry} studied the action of the mapping class group on the curve complex, and then proved the relative hyperbolicity of the mapping class group.

In \cite{baik2018upper}, Baik--Shin--Wu studied fibered 3-manifold groups and related dynamics on the curve complexes. Namely, for a fibered hyperbolic 3-manifold, they showed an estimation for the asymptotic translation lengths of monodromies in a fibered cone, where the asymptotic translation length is measured on the curve complex of each fiber.

In this paper, we extend this result from fibered 3-manifold groups to more general fibered manifold groups. Throughout the paper, we consider smooth connected compact manifolds and simply call them manifolds. For a manifold $M$ and a diffeomorphism $\varphi : M \to M$, we consider the mapping torus $N$ of $\varphi$. This gives a fibration $N \to S^1$ and a flow $\F$ on $N$.
As an analogy of Thurston's fibered cone, we introduce a cone in the first cohomology $H^1(N)$ associated to the monodromy $\varphi : M \to M$, which we call
the \emph{generalized fibered cone} (Definition \ref{def:generalizedfibercone}). We show that each primitive integral element in the generalized fibered cone gives a fibration over the circle with respect to the flow $\F$ (Proposition \ref{prop:gen_fibration}).

Moreover, we prove the characterization of the generalized fibered cone as a rational slice of Fried's cone.
In \cite{fried1982geometry}, for a flow in the manifold, Fried introduced the set of so-called homological directions as a subset of the first homology of the manifold. In the first cohomology, the dual cone of the set of homological directions is called \emph{Fried's cone}, and Fried showed that every primitive integral element corresponds to a fibration over the circle with respect to the given flow.
We also show that the generalized fibered cone is the intersection of a rational subspace and Fried's cone in $H^1(N)$ for the flow $\F$ on $N$ (Theorem \ref{thm.genfibfried_intro}).

In our main theorem, we estimate the asymptotic translation lengths of monodromies coming from the generalized fibered cone, on sphere complexes of fibers (Theorem \ref{theorem:mainthm}).
As an application, we also obtain the precise asymptote of minimal asymptotic translation lengths of pseudo-Anosov handlebody mapping classes on disk graphs in terms of genera (Theorem \ref{theorem.mainhandle}).

\subsection{Generalized fibered cone and asymptotic translation lengths}
For a compact manifold $M$ and a diffeomorphism $\varphi : M \to M$, we consider the mapping torus $N$ of $\varphi$ and define the generalized fibered cone in $H^1(N)$ associated to the monodromy $\varphi : M \to M$ (Definition \ref{def:generalizedfibercone}). We begin with the following characterization of the generalized fibered cone:

\begin{theorem}[Theorem \ref{thm.genfibfried}] \label{thm.genfibfried_intro}
	Let $\varphi : M \to M$ be a diffeomorphism of a compact manifold to itself. Let $N$ be the mapping torus of $\varphi$. Then the generalized fibered cone in $H^1(N)$ associated to $\varphi$ is the intersection of a rational subspace and Fried's cone for the suspension flow given by $\varphi$ in $H^1(N)$.
\end{theorem}

Hence, one can simply regard the generalized fibered cone in the following discussion as a rational slice of Fried's cone. We are mainly interested in the asymptotic translation lengths of monodromies on sphere complexes of fibers. We define the sphere complex as follows, which is a generalization of the curve complex of a surface. A sphere in a compact manifold is called essential if it does not bound a ball or is not isotopic to a boundary component.

\begin{definition}[Sphere complex]
	For a compact manifold $M$ and $n \ge 1$, its \emph{sphere complex} $\mathcal{S}(M; n)$ is a simplicial complex whose vertices are isotopy classes of essential embedded spheres $S^n \subseteq M$, and $k+1$ number of isotopy classes $S_0, \ldots, S_k$ of spheres form a $k$-simplex in $\mathcal{S}(M;n)$ if and only if they can be represented by $k+1$ number of disjoint spheres. 
	
	Each $k$-simplex is identified with the standard simplex in $\R^{k+1}$ spanned by $(1/\sqrt{2})\vec{e}_1, \ldots, (1 / \sqrt{2}) \vec{e}_{k+1}$ where $\vec{e}_i$'s are standard unit vectors. Then we endow the sphere complex $\mathcal{S}(M; n)$ with the induced path metric $d_{\mathcal{S}(M; n)}$.
\end{definition}

We note that for a closed surface $S$, the sphere complex $\mathcal{S}(S;1)$ is the usual \emph{curve complex} of $S$. The 1-skeleton of the curve complex is called the \emph{curve graph} which is quasi-isometrically embedded in the curve complex.

\begin{remark}
	Throughout the paper, most of our argument on the sphere complex $\mathcal{S}(M;n)$ does not depend on the exact value of $n$ while interesting cases are with low codimensions. Hence, when we deal with the sphere complex, we simply use the notation $\mathcal{S}(M)$ to mean by the sphere complex $\mathcal{S}(M;n)$ for some fixed $n$. 
\end{remark}

Sphere complexes have played an important role in geometric group theory and algebraic topology. For instance, the connectivity of various sphere complexes have been obtained and used to show the homological stability of automorphism groups of free groups in \cite{hatcher1995homological}. 

Similar to the curve complex of a surface, a diffeomorphism $M \to M$ naturally induces the isometry $\mathcal{S}(M) \to \mathcal{S}(M)$. Furthermore, two isotopic diffeomorphisms on $M$ induce the same isometry on $\mathcal{S}(M)$. In this regard, we come up with a question pertaining to generalizing the dynamical properties of mapping class groups on curve complexes to the sphere complexes. In particular, we consider asymptotic translation lengths on  sphere complexes.

\begin{definition}[Asymptotic translation length] \label{def:asymptotictl}
	
	Let $(\mathcal{Y}, d)$ be a metric space and $f : \mathcal{Y} \to \mathcal{Y}$ be an isometry. Then its \emph{asymptotic translation length $l_{\mathcal{Y}}(f)$ on $\mathcal{Y}$} is defined as $$l_{\mathcal{Y}}(f) = \liminf_{n \to \infty} {d(f^n(y), y) \over n}$$ for $y \in \mathcal{Y}$.
	
\end{definition}

Throughout the paper, we write $0 \le A(x) \lesssim B(x)$ if there is a constant $C>0$ satisfying $A(x) \le CB(x)$ for all $x$. In addition, when we have both $A(x) \lesssim B(x)$ and $B(x) \lesssim A(x)$, we write $A(x) \asymp B(x)$. Also, since the first cohomology with $\R$-coefficients is a finite-dimensional real vector space, we take any norm $\lVert \cdot \rVert$ without specifying it.
The following is our main theorem. 

\begin{restatable}{theorem}{mainthma} \label{theorem:mainthm}
	Let $\varphi : M\to M$ be a diffeomorphism of a compact manifold to itself. Consider the generalized fibered cone of the mapping torus of $\varphi$, and let $\mathcal{R}$ be an intersection of a $d+1$-dimensional rational subspace 
	and a proper subcone of the generalized fibered cone. Then 
	 we have \begin{equation} \label{eqn.mainest}
	l({\varphi_{\alpha}})\lesssim \lVert \alpha \rVert^{-1-1/d}
	 \end{equation} for all primitive integral element $\alpha \in \mathcal{R}$ where $\varphi_{\alpha}$ is the corresponding monodromy and $l(\varphi_{\alpha})$ is the asymptotic translation length of $\varphi_{\alpha}$ on the sphere complex of the fiber.
\end{restatable}

\begin{remark}
	We simply denote by $l(\cdot)$ the asymptotic translation length on the sphere complex of the underlying manifold since it is of our primary interest. Similarly, in Section \ref{sec:estimate}, we denote $l_{\mathcal{FS}_{\cdot}}(\cdot)$, $l_{\mathcal{FF}_{\cdot}}(\cdot)$, and $l_{\mathcal{D}_{\cdot}}(\cdot)$ for the asymptotic translation lengths of the induced isometries on the free-splitting complex, the free-factor complex, and the disk graph respectively.
	
\end{remark}

We deduce similar statements for asymptotic translation lengths on free-splitting and free-factor complexes of free groups as well. In \cite{dowdall2015dynamics}, Dowdall--Kapovich--Leininger proved that given an expanding irreducible train track map $\psi:G \to G$ which is a homotopy equivalence on a graph $G$, there is an open rational cone, called the {\em positive cone}, in the first cohomology of a (folded) mapping torus of $\psi$ containing the monodromy class. The positive cone is a free-by-cyclic group version of Thurston's fibered cone in the sense that for any primitive integral cohomology class $\alpha$ in the positive cone, $\alpha$ gives an another expanding irreducible train track map $\psi_{\alpha}: G_\alpha\rightarrow G_\alpha$. We note that $\psi_{\alpha}$ may not be uniquely determined by $\alpha$; it requires more choices to be made to obtain $\psi_{\alpha}$ (see \cite[Theorem B]{dowdall2015dynamics} for the precise statement).

As in \cite{aramayona2011automorphisms}, the sphere complex of a 3-manifold with free fundamental group is related to the free-splitting complex, 
(which is also the simplicial completion of the Culler--Vogtmann Outer space \cite{vogtmann2018topology}). Moreover, the barycentric subdivision of the sphere complex with a marked point has something to do with the free-factor complex, as studied by Hatcher--Vogtmann \cite{hatcher1998complex}.
The relation between the free-splitting complex and the free-factor complex was also studied by
Kapovich--Rafi  \cite{kapovich2014hyperbolicity}.
Observing these relations among the sphere complex, the free-splitting complex, and the free-factor complex, 
Theorem \ref{theorem:mainthm} has some implications on the dynamics on free-splitting complexes and  free-factor complexes. 
We deduce analogous estimates on the asymptotic translation lengths on free-splitting complexes and free-factor complexes  from Theorem \ref{theorem:mainthm} (Corollary \ref{cor:FSlength} and Corollary \ref{cor:FFlength}). Together with (\cite{hironaka2011fibered}, \cite{HenselKielak_handlebody}, \cite{laudenbach1974topologie}), these can be interpreted in terms of $\Out(F_g)$-actions of free-splitting and free-factor complexes of the free group $F_g$ (Remark \ref{rmk.out}).

\subsection{Minimal asymptotic translation lengths on disk graphs}

Another application of our results is related to minimal asymptotic translation lengths of subgroups of mapping class groups of surfaces. Let $S_g$ be a closed connected orientable surface of genus $g \ge 2$ and denote its mapping class group by $\Mod(S_g)$.

\begin{definition}[Minimal asymptotic translation length]
	For a subgroup $H \le \Mod(S_g)$ and a metric space $\mathcal{Y}$ on which $H$ isometrically acts, the \emph{minimal asymptotic translation length of $H$ on $\mathcal{Y}$} is $$L_{\mathcal{Y}}(H) := \inf  \{ l_{\mathcal{Y}}(f) : f \in H \text{ is pseudo-Anosov}\}.$$
\end{definition}

Minimal asymptotic translation lengths of some subgroups of mapping class groups have been studied in the settings of Teichm\"uller spaces (e.g. \cite{agol2016pseudo}, \cite{farb2008lower}, \cite{hironaka2011fibered}, \cite{penner1991bounds}) and curve complexes (e.g.  \cite{baik2020minimal}, \cite{baik2018upper}, \cite{gadre2011minimal}, \cite{kin2019small}). On Teichm\"uller space $\mathcal{T}_g$ and curve complex $\C_g$ of $S_g$, the following asymptotes are known for the whole mapping class group $\Mod(S_g)$ and the Torelli group $\I_g$:
\begin{table}[h]
	\begin{center}
	\begin{tabular}{ |c|c|c| } 
	   \hline
	   Subgroups & Teichm\"uller spaces & Curve complexes\\ 
	   \hline
	   $\Mod(S_g)$ &  $\begin{matrix} \\ \mbox{(Penner \cite{penner1991bounds})} \\ L_{\T_g}(\Mod(S_g)) \asymp 1/g  \\ \ \end{matrix}$ & $\begin{matrix} \mbox{(Gadre-Tsai \cite{gadre2011minimal})} \\ L_{\C_g}(\Mod(S_g)) \asymp 1/g^2 \end{matrix}$ \\ 
	   \hline
	   $\I_g$  & $\begin{matrix} \\ \mbox{(Farb-Leininger-Margalit \cite{farb2008lower})} \\ L_{\T_g}(\I_g) \asymp 1 \\ \ \end{matrix}$ & $\begin{matrix} \mbox{(Baik-Shin \cite{baik2020minimal})} \\ L_{\C_g}(\I_g) \asymp 1/g \end{matrix}$ \\ 
	   \hline
	  \end{tabular}
   \end{center}
   \end{table}

   We now consider the \emph{handlebody group} $\H_g < \Mod(S_g)$. That is, identifying $S_g$ with the boundary $\partial V_g$ of genus $g$ handlebody $V_g$, the handlebody group $\H_g$ consists of mapping classes of $S_g$ that extends to $V_g$. Kin--Shin proved in \cite{kin2019small} the following asymptote:
   \begin{equation} \label{eqn.ksasymp}
	L_{\mathcal{C}_g}(\mathcal{H}_g) \asymp {1 \over g^2}.
   \end{equation}

   On the other hand, there is a subcomplex of $\C_g$ on which the handlebody group $\H_g$ acts, the disk graph, which is defined analogous to the curve graph:\footnote{One can indeed consider the disk complex, but for simplicity, we consider the disk graph which is the 1-skeleton of the disk complex.}

   \begin{definition}[Disk graph]
	   \emph{Disk graph} $\mathcal{D}_g$ of the handlebody $V_g$ is a graph whose vertices are isotopy classes of embedded disks $(D^2, \partial D^2) \subseteq (V_g, \partial V_g)$ such that $\partial D^2$ is essential and two vertices are adjacent if they are represented by two disjoint disks. We endow the disk graph with a metric so that each edge is of length 1.
   \end{definition}

   It is clear from the definition that $\H_g$ acts on $\D_g$ by isometries. 
   Masur--Schleimer showed that the disk graph $\D_g$ is Gromov hyperbolic \cite{MasurSchleimer_disk}. The inclusion $(D^2, \partial D^2) \subset (V_g, \partial V_g)$ induces an embedding $\D_g \to \C_g$, and the image of $\D_g$ under this embedding is quasi-convex in $\C_g$ as shown by Masur--Minsky  \cite{MasurMinsky04}. However, the disk graph $\D_g$ is distorted in the curve complex $\C_g$. Indeed, it is not quasi-isometrically embedded \cite{MasurSchleimer_disk}. Hence it is not straightforward whether  handlebody groups $\H_g$ have the same asymptote for the minimal asymptotic translation lengths on disk graphs $\D_g$ and on curve complexes $\C_g$. As an application of Theorem \ref{theorem:mainthm}, we answer the affirmative:

   \begin{restatable}{theorem}{mainthmhandle} \label{theorem.mainhandle}
	We have $$L_{\D_g}(\H_g) \asymp \frac{1}{g^2}.$$

   \end{restatable}

   Note that the lower bound of the above asymptote can be deduced from \eqref{eqn.ksasymp}.

\subsection{Future directions} 
In \cite{culler1986moduli}, Culler--Vogtmann introduced the notion of the Outer space $CV_g$ which is equipped with a natural action of $\Out(F_g)$. Roughly speaking, $CV_g$ is the space of marked metric graph structures on $F_g$ of volume 1. It has a natural simplicial decomposition in terms of graphs and the vertices that can be re-interpreted as splittings of $F_g$ as a free product or HNN extension via Bass--Serre theory \cite{vogtmann2018topology}. This allows us to identify the free-splitting complex $\mathcal{FS}_g$ with the simplicial closure of $CV_g$ as we noted earlier. For a general review on the geometry of Outer space, one can refer to \cite{vogtmann2015geometry}.

According to \cite{handel2019free}, a fully irreducible element in $\Out(F_g)$ acts as a hyperbolic isometry on $\mathcal{FS}_g$ which is equivalent to the sphere complex. Hence, we can say more if we could figure out a lower bound of translation lengths in Theorem \ref{theorem:mainthm} or Corollary \ref{cor:FSlength}. To be precise, let $\phi \in \Out(F_g)$ be fully irreducible and let $L$ be its quasi-axis on $\mathcal{FS}_g$. Then $L$ and a geodesic connecting $x \in L$ and $\phi(x)$ pass through coarsely. Noting that the Outer space can also be defined via sphere systems in a doubled handlebody as introduced in \cite{hatcher1995homological}, the lower bound for translation lengths on $\mathcal{FS}_g$ gives a lower bound for the number of foldings one needs to get $\phi(x)$ from $x$ as points in $CV_g$, by comparing the barycentric subdivision and the dual complex of $\mathcal{FS}_g$.
Moreover, it would be possible to make the lower bounds uniform on a positive cone in \cite{dowdall2015dynamics} if we could further control various sphere complexes and monodromies from the positive cone. Indeed, if $\phi \in \Out(F_g)$ further satisfies that $F_g \rtimes_{\phi}\Z$ is word-hyperbolic, monodromies from the positive cone for $\phi$ are fully irreducible by \cite{dowdall2015dynamics}.

Another important remark is that the deductions of Theorem \ref{thm:positivecone} from Theorem \ref{theorem:mainthm} are based on the concrete geometric description of the generalized fibered cone for doubled handlebody case. However, for general manifolds other than surfaces and doubled handlebodies, we do not have a concrete description of the generalized fibered cone. Finding a concrete geometric description of the generalized fibered cone in more general setting would lead to many other applications of our approach. 

Lastly, Baik--Kin--Shin--Wu conjectured in \cite{baik2019asymptotic} that if an element of the mapping class group of a surface has small asymptotic translation length on the curve complex, then the element is a normal generator of the mapping class group. When one considers the action on the Teichm\"uller space instead of the curve complex, such a phenomenon was obtained by Lanier--Margalit \cite[Theorem 1.2]{lanier2018normal}. One might view our present article as a beginning step toward an analogous question for $\Out(F_g)$, or more generally, automorphism groups of sphere complexes.

\subsection*{Organization}

We define the generalized fibered cone and characterize it as a rational slice of Fried's cone in Section \ref{sec:generalizedfiberedcone}, proving Theorem \ref{thm.genfibfried_intro}. In Section \ref{sec:mainA}, we prove Theorem~\ref{theorem:mainthm}. Section \ref{sec.handlebody} is devoted to the application to minimal asymptotic translation lengths of handlebody groups on disk graphs (Theorem \ref{theorem.mainhandle}). We provide a construction of a diffeomorphism between doubled handlebodies from a given folding sequence for a combinatorial map between graphs in 
Section \ref{sec:3mnfd}. In Section \ref{sec:estimate}, we discuss applications to free-splitting complexes and free-factor complexes of free groups.

\subsection*{Acknowledgements}

We greatly appreciate for Sebastian Hensel, Autumn Kent, Daniel Levitin, Karen Vogtmann, and Nathalie Wahl for many helpful discussions and comments. We give our special thanks to Camille Horbez for reading the first version of the draft and suggesting the proof of Corollary \ref{cor:FFlength}. Finally we thank the anonymous referee for careful reading and many valuable comments.

The first author was partially supported by the National Research Foundation of Korea(NRF) grant funded by the Korea government(MSIT) (No. 2020R1C1C1A01006912)


\section{Generalized fibered cone and slices of Fried's cone} \label{sec:generalizedfiberedcone}

Recall that a fibered cone for the surface case is a cone in the first cohomology of a fibered hyperbolic 3-manifold such that every primitive integral cohomology class in the cone corresponds to a fibration over the circle. In this section, we define a higher-dimensional analogue of the fibered cone, the \emph{generalized fibered cone}. 

Let $M$ be a compact manifold and $\varphi: M \to M$ be a diffeomorphism. 
Let $\Z^d \cong H \le H_1(M)$ be a free abelian subgroup invariant under $\varphi$ with the connected free abelian cover $\tilde{M}$ having Deck group  $H$. We fix a lift $\tilde{\varphi} : \tilde{M} \to \tilde{M}$ of $\varphi : M \to M$.
 In other words, we have the following commutative diagram: $$\mbox{\color{white} Deck group }\begin{tikzcd}
\tilde{M} \arrow[r, "\tilde{\varphi}"] \arrow[d] & \tilde{M} \arrow[d, "\mbox{Deck group }H"] \\
M \arrow[r, "\varphi"] & M
\end{tikzcd}$$

 Let $N := M \times \R / (x, s) \sim (\varphi(x), s + 1)$ be the mapping torus.
We also set $\tilde{N}' :=\tilde{M}\times\mathbb{R}$ 
and $\tilde{N} :=\tilde{N}'/(x, s)\sim (\tilde{\varphi}(x), s+1)$ which is the mapping torus of $\tilde{\varphi}$. The associated flow $\F := \{\F_t\}_{t \in \R}$ on $N$ is given by the projection of the map $(x, s) \mapsto (x, s + t)$ on $\tilde N'$ at time $t \in \R$. We summarize the relationship among these spaces using the following commutative diagram:
$$\begin{tikzcd}
\tilde{M} \arrow[r, "\tilde{\varphi}"] \arrow[d]& \tilde{M} \arrow[d] \\
M \arrow[r, "\varphi"] & M
\end{tikzcd} \quad \xrightarrow{\mbox{ mapping torus }} \quad \begin{tikzcd}
\tilde{N} \arrow[d] & \tilde{N}' \arrow[l, "\ \Z- \mbox{fold} \ "] \arrow[ld, dashed, "\mbox{Deck group } \Gamma"] \\
N & \end{tikzcd}$$
Then $\tilde{N}'$ is a $\Gamma := H \oplus \Z$ cover of $N$. Here $(h, n)\in \Gamma$ applied to $(x, s)\in \tilde{N}'$ is $(\tilde{\varphi}^nh(x), s+n)$. In other words, $\Gamma = H \oplus \Z$ is a quotient of $H_1(N)$, hence its dual $\Gamma^* = \Hom(\Gamma, \Z)$ is a subgroup of $H^1(N)$. Throughout the paper, we also use the coordinate in $\Gamma^*$ dual to $\Gamma$. Furthermore, 
the $\Gamma$-action is restricted to the action on $\tilde M$ given by $(h, n)\cdot x = \tilde \varphi^n h(x)$ for $x \in \tilde M$. In this point of view, we sometimes identify $(h, n) = \tilde \varphi^n h$ when we discuss the action on $\tilde M$.
By abusing notations, we use multiplication for the group operation on $H$ when we consider elements of $H$ as maps, and use addition when we regard $H$ as a free abelian group.

Now fix a fundamental domain $D$ of the cover $\tilde{M} \to M$. For a map $f : \tilde{M}\to \tilde{M}$, let us define $$\Omega(f) := \mathrm{CH} \left \lbrace h \in H : (h \cdot D) \cap f(D) \neq \emptyset \right \rbrace$$ where $\mathrm{CH}\{ \cdot \}$ denotes a convex hull of $\{\cdot \}$ in $H \otimes \R$. Using this notation, we define $$\Omega := \bigcup_{t\in\Z} \left( - \Omega \left( \tilde{\varphi}^t \right)\times \{t\} \right) \subseteq \Gamma \otimes \R.$$ 
We then have:

\begin{prop} \label{prop:non_empty}
  The set $\Omega$ is contained in a pair of rational cones whose intersection with $H\otimes\mathbb{R}\times \{\pm 1\}$ is bounded.
\end{prop}

\begin{proof}
  Let $C$ be the convex hull of $-\Omega(\tilde{\varphi})\cup \Omega(\tilde{\varphi}^{-1})$. Then $C$ is a bounded polygon because $D$ is compact. Then, because for any $n>0$,
  \[\begin{aligned}
  \Omega(\tilde{\varphi}^n)& =CH\{h\in H : (h\cdot D)\cap \tilde{\varphi}^n(D)\not=\emptyset\} \\
  & =CH\left\lbrace h\in H : \begin{matrix}
  \exists h_1 \text{ such that }(h\cdot D)\cap\tilde{\varphi}(h_1\cdot D)\not=\emptyset,\\ (h_1\cdot D)\cap \tilde{\varphi}^{n-1}(D)\not=\emptyset
  \end{matrix}\right\rbrace
  \end{aligned}\]
  We have $\Omega(\tilde{\varphi}^n)\subset (-C)+(-C)+\cdots+(-C)$, where the addition is among $n$ copies of $(-C)$. Here, we use conventions $-A=\{x\in H\otimes\mathbb{R}: -x\in A\}$ and  $A+B=\{a+b: a\in A, b\in B\}$.
  Similarly, $\Omega(\tilde{\varphi}^{-n})$ is contained in the sum of $n$ copies of $C$. As a consequence, $\Omega$ is contained in the set $\{(x, t): x=0, t=0\text{ or }x/t\in C\}$.
\end{proof}

We denote by $\hat \Omega \subset \Gamma \otimes \R$ the asymptotic cone of $\Omega$, i.e., $$\hat \Omega := \{ x \in \Gamma \otimes \R : x = \lim_{i \to \infty} t_i w_i \text{ for some sequences } w_i \in \Omega \text{ and } t_i \to 0\}.$$
Now we define the \emph{generalized fibered cone}:

\begin{definition}[Generalized fibered cone] \label{def:generalizedfibercone}
	In the above setting of a diffeomorphism $\varphi:M \to M$ on a compact manifold $M$ with a choice of a fundamental domain, the \emph{generalized fibered cone} is the dual cone $\C \subset H^1(N)$ of the asymptotic cone of $\Omega$:
	 for the asymptotic cone $\hat \Omega$ of $\Omega$,
	$$\C := \{ \alpha \in \Gamma^* \otimes \R : \sign(t)\alpha(h, t) > 0 \text{ for all non-zero } (h, t) \in \hat \Omega \}$$ where $\sign(t)=1$ when $t>0$ and $\sign(t)=-1$ when $t<0$. In particular, for any $\alpha \in \C$, there exists $K > 0$ such that for any $(h, t) \in \Omega$ with $|t| > K$, we have $\sign(t) \alpha(h, t) > 0$.
      \end{definition}

	 By Proposition \ref{prop:non_empty},  the generalized fibered cone always has non-empty interior. Moreover,  if $\Omega \subset \Gamma \otimes \R$ is of finite Hausdorff distance to the union of two cones in $\Gamma \otimes \R$ centered at the origin, the first one with $t\geq 0$ and second one with $t\leq 0$, then $\mathcal{C}$ is the intersection of the dual cone of the first cone and the negative of the dual cone of the second one.   Although the choice of fundamental domain is involved in defining the generalized fibered cone, it is independent of the choice:

	  \begin{lem}
	  The generalized fibered cone does not depend on the choice of the fundamental domain $D$.
	  \end{lem}

	  \begin{proof}
	  This is due to the fact that the covering $\tilde M \to M$ is abelian. Let $D$ and $D'$ be fundamental domains with compact closures, and fix an $H$-equivariant map $f : \tilde M \to \tilde M$. If $(h \cdot D) \cap f(D) \neq \emptyset$ for $h \in H$, we then have $x_1, x_2 \in D$ such that $h(x_1) = f(x_2)$. Since $D'$ is also a fundamental domain, there exist $h_1, h_2 \in H$ and $x_1', x_2' \in D'$ such that $x_1 = h_1 x'$ and $x_2 = h_2 x_2'$. Hence, we have $(h h_1)(x_1') = (f h_2)(x_2')$. By the equivariance of $f$ and the fact that $H$ is abelian, we have $(h_2^{-1}h_1 h)(x_1') = f(x_2')$. Since both $D$ and $D'$ have compact closures, there are only finitely many possible $h_1$ and $h_2$, depending only on $D$ and $D'$. This implies that the set $\Omega(f)$ obtained using $D$ is within bounded Hausdorff distance from the one obtained using $D'$, and vice versa. It then follows that the generalized fibered cone is independent of the choice between $D$ and $D'$.
	  \end{proof}

\subsection{Generalized fibered cone as the slice of Fried's cone}

Fried showed in \cite{fried1982geometry} that given a flow on a compact manifold, there exists a cone in the first cohomology so that a primitive integral element gives a fibration over the circle with respect to the given flow if and only if it belongs to the cone. Applying this result to the flow $\F$ on $N$, we obtain the cone, $\C_F \subset H^1(N)$, which we call Fried's cone. Our generalized fibered cone is indeed a slice of Fried's cone.

\begin{thm}[Theorem \ref{thm.genfibfried_intro}] \label{thm.genfibfried}
	We have $$\C = \C_F \cap (\Gamma^* \otimes \R).$$
	In other words, a primitive integral element $\alpha \in \Gamma^*$ gives a fibration $N \to S^1$ with respect to the flow $\F$ if and only if $\alpha \in \C$.
\end{thm}

To prove this, we first recall the work of Fried \cite{fried1982geometry}: let $X$ be a compact smooth manifold and $\phi = (\phi_t)_{t \in \R}$ be a $C^1$-flow on $X$ which is either transverse or tangent to each component of $\partial X$. Let $\tilde X$ be a connected $\Z$-cover with Deck transformation $g : \tilde X \to \tilde X$ and a lift $\tilde \phi$ of the flow $\phi$. We denote by $\tilde X/\tilde \phi$ the space of $\tilde \phi$-flow lines in $\tilde X$. The $\Z$-cover $\tilde X$ can be compactified by two points $\{\pm \infty\}$ so that $g^n x \to \pm \infty$ as $n \to \pm \infty$ for any $x \in \tilde X$. Fried gave the following characterization of the cross section of $\phi$ in terms of the behavior of $\tilde \phi$:

\begin{thm} \cite[Theorem A]{fried1982geometry} \label{thm.fried}
In the above setting, the following are equivalent:
\begin{enumerate}
	\item for any $x \in \tilde X$, $\tilde \phi_t(x) \to \pm \infty$ as $t \to \pm \infty$.
	\item  $K := \tilde X / \tilde \phi$ is a cross section of the flow $\phi$ in $X$ so that we can identify $\tilde X = K \times \R$  and have $\tilde \phi_t(k, s) = (k, s + t)$ and $g(r(k), s) = (k, s + t(k))$ where $t(k)$ is the return time of $k \in K$ under $\phi$ and $r(k) = \phi_{t(k)}(k)$. In particular, $X$ is fibered over the circle with fiber $K = \tilde X / \tilde \phi$.
\end{enumerate}
\end{thm}

Now recall that we have a mapping torus $N$ of $M$ with the monodromy $\varphi$. Using the Fried's result, we show the following proposition, which implies that $\C \subset \C_F \cap (\Gamma^* \otimes \R)$. For $\alpha \in \Gamma^*$, we denote by $\alpha^{\perp} < \Gamma$ the subgroup consisting of elements whose pairing with $\alpha$ is $0$.

\begin{prop} \label{prop:gen_fibration}
  Let $\alpha=(\cdot, n_{\alpha})$ be any primitive integral class in $\mathcal{C}$. Then $N$ admits another fibration over the circle respecting the flow $\F$ so that the generator of the first cohomology of the circle pulls back to $\alpha$, and the fiber is
  \[M_\alpha=\tilde{M}/\alpha^{\perp}.\]
Moreover, the monodromy $\varphi_{\alpha}$ has a lift $\tilde \varphi_{\alpha}$ on $\tilde M$ such that $\tilde \varphi_{\alpha}^{n_{\alpha}} = \tilde \varphi$.
\end{prop}
Here, a fibration over the circle respecting the flow $\F$ refers to a fibration such that the suspension flow is conjugate to the flow $\F$. 

\begin{proof}

  To show the desired fibration, recall the commutative diagram of coverings: $$\begin{tikzcd}
	\tilde{N} \arrow[d] & \tilde{N}'=\tilde{M}\times\mathbb{R} \arrow[l, "\ \Z- \mbox{fold} \ "] \arrow[ld, dashed, "\mbox{Deck group } \Gamma = H \oplus \Z"] \\
	N & \end{tikzcd}$$	
	and the flow $\F$ can be lifted to the $\R$-translation flow on $\tilde N' = \tilde M \times \R$. 
	We consider the $\Z$-cover $$\tilde N'/ \alpha^\perp \to N$$
	and denote by $\tilde \F$ the induced flow on $\tilde N' / \alpha^{\perp}$, which is the lift of $\F$ as well. The Deck group is generated by $(g, m) \in H \oplus \Z$ such that $\alpha(g, m) = 1$. Since the set of $\tilde \F$-flow lines in $\tilde N' / \alpha^\perp$ is identified with $\tilde M / \alpha^\perp$, the first statement follows once we verify that Theorem \ref{thm.fried}(1) holds with $X = N$, $\tilde X = \tilde N' / \alpha^{\perp}$, $\phi = \F$, and $\tilde \phi = \tilde \F$.

	We first claim that any $\tilde \F$-flow line in $\tilde N' / \alpha^{\perp}$ does not accumulate. Suppose not. It means that there exist $(x, s) \in \tilde N' = \tilde M \times \R$ and sequences $t_i \to \pm \infty$ and $(h_i, n_i) \in \alpha^{\perp} < H \oplus \Z$ such that the sequence $$(h_i, n_i)\cdot(x, s + t_i) = (h_i \tilde \varphi^{n_i}(x), s + n_i + t_i)$$ converges in $\tilde N'$. Hence, the sequence $n_i$ is divergent and $h_i \tilde \varphi^{n_i}(x)$ converges in $\tilde M$. In particular, there exists $h \in H$ such that $hh_i\tilde \varphi^{n_i}(D) \cap D \neq \emptyset$ for all $i$. This implies that $h h_i \in - \Omega(\tilde \varphi^{n_i})$, and therefore $(h_i, n_i)$ is contained in a bounded neighborhood of $\Omega$. Since $n_i$ is a divergent sequence and $\alpha(h_i, n_i) = 0$, $\alpha$ vanishes at some vector in the asymptotic cone of $\Omega$. This contradicts $\alpha \in \C$.

	Now to verify Theorem \ref{thm.fried}(1), let $(x, s) \in \tilde N' = \tilde M \times \R$. For each $t \in \R$, there exist $(h_t, n_t) \in \alpha^{\perp}$ and $k_t \in \Z$ such that $$(g, m)^{k_t}(h_t, n_t)\cdot (x, s + t) = (g^{k_t}h_t \tilde \varphi^{k_t m + n_t}(x), s + t + k_t m + n_t)$$ is contained in a fixed compact subset. By the previous claim, as $t \to \infty$, we have either $k_t \to \infty$ or $k_t \to - \infty$. After passing to a subsequence, we may assume that the sequence $$(g^{k_t}h_t \tilde \varphi^{k_t m + n_t}(x), s + t + k_t m + n_t)$$ converges in $\tilde N'$. This implies that 
	\begin{equation} \label{eqn.negativeone}
		\frac{k_t m + n_t}{t} \to -1
	\end{equation} as $t \to \infty$ and for some $h \in H$, we have $$hg^{k_t}h_t \tilde \varphi^{k_t m + n_t}(D) \cap D \neq \emptyset$$ for all large $t > 0$. Hence $(g, m)^{k_t} (h_t, n_t) = (g^{k_t}h_t, k_t m + n_t)$ is of bounded distance from some point $(y_t, k_t m + n_t) \in - \Omega (\tilde \varphi^{k_t m + n_t}) \times \{ k_t m + n_t\}$. Therefore,
	$$\lim_{t \to \infty} \frac{k_t}{k_t m + n_t} = \lim_{t \to \infty} \frac{\alpha((g, m)^{k_t}(h_t, n_t))}{k_t m + n_t} = \lim_{t \to \infty} \alpha \left( \frac{y_t}{k_t m + n_t}, 1 \right).$$
	Since the vector $ \left( \frac{y_t}{k_t m + n_t}, 1 \right)$ converges to the vector in the asymptotic cone of $\Omega$ and $\alpha \in \C$, we obtain
	$$\lim_{t \to \infty} \frac{k_t}{k_t m + n_t} > 0.$$ Together with \eqref{eqn.negativeone}, we have that $k_t \to - \infty$ as $t \to \infty$. Note that the positivity of $t$ has not been used. Applying the same argument to the case $t \to - \infty$, we conclude that $k_t \to \mp \infty$ as $t \to \pm \infty$. Consequently, in the compactification of $\tilde N'/\alpha^{\perp}$, each $\tilde \F$-flow line is from the end $(g, m)^{-\infty}$ to the end $(g, m)^{\infty}$.
	This verifies Theorem \ref{thm.fried}(1), showing the first statement of the proposition.

	The last assertion follows from the observation that $\alpha(0, 1) = n_{\alpha}$, and hence $(g, m)^{n_{\alpha}}(0, -1) \in \alpha^{\perp}$. 
\end{proof}

The converse might be standard to experts, but we present the proof as follows, completing the proof of Theorem \ref{thm.genfibfried}:

\begin{prop}
	Let $\alpha \in \Gamma^*$ be a primitive integral class. If $N$ admits a circle fibration respecting the flow $\F$ associated to $\alpha$, then $\alpha \in \C$.
\end{prop}

\begin{proof}
Again, we consider the $\Z$-cover $$\tilde N' / \alpha^{\perp} \to N.$$ Then from the hypothesis, the flow $\F$ admits a cross section that can be lifted to $\tilde N' / \alpha^{\perp}$. Recall the commutative diagram:
$$\begin{tikzcd}
	\tilde{N} \arrow[d] & \tilde{N}'=\tilde{M}\times\mathbb{R} \arrow[l, "\ \Z- \mbox{fold} \ "] \arrow[ld, dashed, "\mbox{Deck group } \Gamma = H \oplus \Z"] \\
	N & \end{tikzcd}$$
	We also denote by $\tilde \F$ and $\tilde \F'$ the lifts of the flow $\F$ on $\tilde N' / \alpha^{\perp}$ and $\tilde N'$ respectively.

Let $h_n \in H$, $n \in \Z$, be a sequence such that $$(h_n \tilde \varphi^n \cdot D) \cap D \neq \emptyset.$$
To prove $\alpha \in \C$, it suffices to show that the sequence $\alpha(h_n, n)/n$ is a positive for all but finitely many $n$, and does not accumulate to $0$. Fix an element $(g, m) \in \Gamma$ such that $\alpha(g, m) = 1$. This acts as a Deck transformation on $\tilde N' / \alpha^{\perp}$.

For each $n \in \Z$, let $x_n \in D$ be such that $h_n \tilde \varphi^n(x_n) \in D$.  Then for each $n \in \Z$, there exists $k_n \in \Z$ and $(g_n, t_n) \in \alpha^{\perp}$ such that $$(g, m)^{k_n}(g_n, t_n)(h_n, n)\cdot \tilde \F'_{-n}(x_n, 0) \in \tilde N'$$ is contained in the fundamental domain for the $\Gamma$-action on $\tilde N'$ containing $D \times \{0\}$.
Expanding the action, this means that
$$(g^{k_n}g_n \tilde \varphi^{k_n m + t_n} h_n \tilde \varphi^{n}(x_n), k_n m + t_n) \in \tilde N'$$ belongs to the $\Gamma$-fundamental domain on $\tilde N'$. In particular, the sequence $k_n m + t_n \in \Z$ is bounded. Since $h_n \tilde \varphi^{n}(x_n) \in D$ for all $n \in \Z$, we now have that the sequence $\tilde \varphi^{k_n m + t_n} h_n \tilde \varphi^{n}(x_n)$ is bounded in $\tilde M$ as well. Together with the boundedness of the sequence $g^{k_n}g_n \tilde \varphi^{k_n m + t_n} h_n \tilde \varphi^{n}(x_n) \in \tilde M$, it follows that the sequence $g^{k_n}g_n \in H$ is bounded. Therefore, the sequence $(g, m)^{k_n}(g_n, t_n) = (g^{k_n} g_n, k_n m + t_n) \in \Gamma$ is bounded.

Now for each $n \in \Z$, let $(z_n, s_n) = (g, m)^{k_n}(g_n, t_n)(h_n, n)\cdot \tilde \F'_{-n}(x_n, 0) \in \tilde N'$.
We then have
$$(g, m)^{k_n}(g_n, t_n)(h_n, n)\cdot (x_n, 0) = \tilde \F'_{n}(z_n, s_n) \in \tilde N'.$$
If we denote by $\tilde \alpha'$ the 1-form on $\tilde N'$ induced by $\alpha$, we have
$$\alpha((g, m)^{k_n}(g_n, t_n)(h_n, n)) = \int_{(x_n, 0)}^{(z_n, s_n)} \tilde \alpha' + \int_{(z_n, s_n)}^{\tilde \F_{n}'(z_n, s_n)} \tilde \alpha'.$$
Since $(x_n, 0)$ and $(z_n, s_n)$ are contained in the fixed fundamental domain of the $\Gamma$-action on $\tilde N'$, $\int_{(x_n, 0)}^{(z_n, s_n)} \tilde \alpha'$ is bounded. Moreover, since the return time for the flow $\F$ to the fiber $\tilde M / \alpha^{\perp}$ is bounded from below and above by positive constants, there exist $c, c' > 1$ such that
$$\begin{aligned}
c^{-1} n - c' & \le \alpha((g, m)^{k_n}(g_n, t_n)(h_n, n)) \le c n + c' & \quad \text{for } n \ge 0 \\
c n - c' & \le \alpha((g, m)^{k_n}(g_n, t_n)(h_n, n)) \le c^{-1} n + c' & \quad \text{for } n < 0.
\end{aligned}$$
Since the sequence $(g, m)^{k_n}(g_n, t_n)$ is bounded, this finishes the proof.
\end{proof}

The following lemma will be used later:

\begin{lem} {\cite[Lemma 5.3]{baik2018upper}} \label{lem:geomarg}
Let $\mathcal{C}_0$ be a proper subcone of the generalized fibered cone. There exists $C > 0$ such that for any primitive integral element $\alpha = (\cdot, n_{\alpha}) \in \mathcal{C}_0$ with $n_{\alpha} > C$, there is some $h \in H$ which does not belong to the $C n_{\alpha}^{1/d}$-neighborhood of $\bigcup_{a \in \alpha^{\perp}} \Omega(a)$ in $H$.
\end{lem}

\begin{proof}
	The definition of generalized fibered cone implies that $\Omega$ is contained in a Hausdorff neighborhood of the dual cone of $\mathcal{C}_0$. Let $\alpha$ be a primitive integral element in $\mathcal{C}_0$, $b = (x_1, \ldots, x_d, y) \in \alpha^{\perp}$, and $p \in \Omega(b)$. Then Proposition \ref{prop:non_empty} implies $d(\pm \Omega(\tilde{\varphi}^y), 0) \le Ay + C$ for some $A, C > 0$. Furthermore, the fact that $\Omega$ is contained in a Hausdorff neighborhood of $\mathcal{C}_0$ implies that $x := (x_1, \ldots, x_d)$ satisfies $d(x, -\Omega(\tilde{\varphi}^y)) \ge A'y - C'$ for some $A', C' > 0$. This observation implies the same statement as \cite[Lemma 5.1]{baik2018upper}. Together with the fact that the covolume of the projection of $\alpha^{\perp}$ onto $H$ is $\gtrsim |n_{\alpha}|$ (cf. \cite[Lemma 5.2]{baik2018upper}), the same proof as in \cite[Lemma 5.3]{baik2018upper} works and gives the desired statement.
\end{proof}


\section{Asymptotic translation lengths on sphere complexes}\label{sec:mainA}

Now we prove the main theorem:

\mainthma*

We deduce it from the following weaker version:
\begin{thm} \label{thm.weaker}
Let $\varphi : M\to M$ be a diffeomorphism of a compact manifold to itself. Consider the generalized fibered cone of the mapping torus of $\varphi$, and let $\mathcal{R}$ be an intersection of a $d+1$-dimensional rational subspace 
	and a proper subcone of the generalized fibered cone. Then there exist finitely many hyperplanes $R_1, \cdots, R_k$, disjoint from the cohomology class corresponding to $\varphi$, such that \[
	l({\varphi_{\alpha}})\lesssim \lVert \alpha \rVert^{-1-1/d}
	\] for all primitive integral element $\alpha \in \mathcal{R} - \bigcup_{i = 1}^k R_i$ where $\varphi_{\alpha}$ is the corresponding monodromy.
\end{thm}

\begin{proof}
By Lemma \ref{lem:geomarg}, there exists $C > 0$ depending on $\mathcal{R}$ such that for any primitive integral element $\alpha = (\cdot, n_{\alpha}) \in \mathcal{R}$ in coordinates of $\Gamma^* = (H \oplus \Z)^*$ with $n_{\alpha} > C$, there exists $h \in H$ such that the $C n_{\alpha}^{1/d}$-neighborhood of $h$ in $H$ is disjoint from $\bigcup_{a \in \alpha^{\perp}} \Omega(a)$. In particular, $h$ is disjoint from $\bigcup_{a \in \alpha^{\perp}} \Omega(a)$ and hence 
\begin{equation} \label{eqn.disjoint}
(h \cdot D) \cap (\alpha^{\perp} \cdot D) = \emptyset
\end{equation}

We claim that we can choose $c > 0$ small enough so that for any such $\alpha$ and $h$ above, we have
\begin{equation} \label{eqn:mainthm}
	\left(\left( h \tilde{\varphi} ^{\left\lfloor cn_{\alpha}^{1/d}\right\rfloor} \right) \cdot D\right) \cap (\alpha^{\perp} \cdot D) = \emptyset.
	\end{equation} 
	See Figure \ref{fig:mainthm} for the pictorial description of the claim.
	\begin{figure}[h]
		\begin{tikzpicture}[scale=1.5, every node/.style={scale=0.9}]
		
		\draw[pattern=dots] (-1, -1) -- (-1.8, -0.5) -- (-2, 0) -- (0.3, -0.7) -- (-1, -1);
		\draw (-1.4, -0.8) node[left] {$\Omega(a_1)$};
		
		\draw[dashed, fill=white!8] (2.5, -0.5) arc(0:360:1);
		
		\draw[dashed, fill = white!20] (2.5, -0.5) .. controls (2.5, -0.2) and (2, -0.05) .. (1.5, -0.05) .. controls (1, -0.05) and (0.5, -0.2) .. (0.5, -0.5) .. controls (0.5, -0.7) and (1, -0.95) .. (1.5, -0.95) .. controls (2, -0.95) and (2.5, -0.7) .. (2.5, -0.5);

		\draw[->] (1.5, -0.5) -- (2.5, -0.5);
		\draw (2, -0.55) node[above] {$C n_{\alpha}^{1/d}$};
		
		\draw[pattern=dots] (1, -1.3) -- (3, -0.5) -- (3, -1.1) -- (2.5, -1.5) -- (2, -1.5) -- (1, -1.3);
		\draw (3, -1) node[right] {$\Omega(a_2)$};
		
		\draw[->, thick] (0, 0) -- (0, 1);
		\draw (0, 1) node[left] {$\tilde{\varphi}$};

		\draw[->] (2, 1) .. controls (1.8, 1) and (1.75, 0.9) .. (1.75, 0.6) .. controls (1.75, 0.3) and (1.7, 0.2) .. (1.55, 0.2);

		\draw (2, 1.1) node[right] {$h\tilde{\varphi}^{\left \lfloor cn_{\alpha}^{1/d} \right \rfloor}$};
		
		\draw[->, thick] (0, 0) -- (3, 0);
		\draw[->, thick] (0, 0) -- (-1, -1.8);
		\draw (3, -1.8) node[left] {$H \otimes \R$};
		
		\draw (1.5, 0.2) -- (1.5,-0.5);
		\draw (1.5, -0.4) -- (1.6, -0.4) -- (1.6, -0.5);
		\filldraw (1.5, 0.2) circle(1pt);
		\filldraw (1.5, -0.5) circle(1pt);
		\draw (1.5, -0.5) node[left] {$h$};
		
		\draw (0, -2);

		\end{tikzpicture}
		\caption{Description of $\Gamma \otimes \R$. The subspace $H \otimes \R$ is illustrated as a horizontal plane. $a_1, a_2 \in \alpha^{\perp}$ and the dotted regions are $\Omega(a_1)$ and $\Omega(a_2)$. The ball centered at $h$ is of radius $Cn_{\alpha}^{1/d}$ in $\Gamma \otimes \R$. The constant $c$ is chosen appropriately so that $h\tilde{\varphi}^{\left\lfloor cn_{\alpha}^{1/d}\right\rfloor}$ belongs to the ball.} \label{fig:mainthm}
		\end{figure}
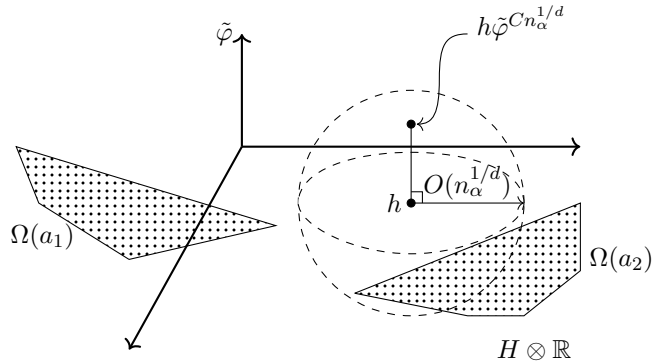

	Let $c > 0$ be a constant. We prove the claim by showing that we can take $c$ small enough so that for any $\alpha$ and $h$ as above, we have 
	\begin{equation} \label{eqn:equivdisj} 
		h \cdot D \cap \left(\alpha^{\perp} \tilde{\varphi}^{-\left\lfloor cn_{\alpha}^{1/d}\right\rfloor}\right) \cdot D = \emptyset.
	\end{equation}
	Fix such $\alpha$ and $h$, and suppose that 
	\begin{equation} \label{eqn.detailclaim}
		h \cdot D \cap \left(a \tilde{\varphi}^{-\left\lfloor cn_{\alpha}^{1/d}\right\rfloor}\right) \cdot D \neq \emptyset
	\end{equation}
	for some $a \in \alpha^{\perp}$. Writing $a = x \tilde \varphi^m$ for some $x \in H$ and $m \in \Z$, we have
	$$\left( x^{-1}h \cdot D \right) \cap \left( \tilde{\varphi}^{m -\left\lfloor cn_{\alpha}^{1/d}\right\rfloor} \cdot D \right)  \neq \emptyset.$$
	Then there exists $y \in \tilde{\varphi}^{ -\left\lfloor cn_{\alpha}^{1/d}\right\rfloor} \cdot D$ such that $\tilde \varphi^m y \in x^{-1} h \cdot D$. We can also choose $h_1 \in H$ such that $y \in h_1 \cdot D$. This implies that
	$$\left( x^{-1} h \cdot D\right) \cap \left( \tilde \varphi^m h_1 \cdot D \right) \neq \emptyset \quad \text{and} \quad \left( h_1 \cdot D \right) \cap \left( \tilde{\varphi}^{ -\left\lfloor cn_{\alpha}^{1/d}\right\rfloor} \cdot D \right) \neq \emptyset.$$
	In particular, $h_1^{-1}x^{-1}h \in \Omega(\tilde \varphi^m)$ and $h_1 \in \Omega \left(  \tilde{\varphi}^{ -\left\lfloor cn_{\alpha}^{1/d}\right\rfloor} \right)$.
	Therefore, we have
	$$h \in x \Omega(\tilde \varphi^m) \Omega \left(  \tilde{\varphi}^{ -\left\lfloor cn_{\alpha}^{1/d}\right\rfloor} \right).$$
	Since $x \Omega(\tilde \varphi^m) = \Omega(a)$ and the diameter of $\Omega  \left(  \tilde{\varphi}^{ -\left\lfloor cn_{\alpha}^{1/d}\right\rfloor} \right)$ is less than $c_0 c n_{\alpha}^{1/d}$ for some constant $c_0 > 0$ depending only on $\mathcal{R}$ by Proposition \ref{prop:non_empty}, we now have that $h$ is contained in the $c_0 c n_{\alpha}^{1/d}$-neighborhood of $\Omega(a)$ in $H$. On the other hand, $h$ was chosen so that the $C n_{\alpha}^{1/d}$-neighborhood of $h$ is disjoint from $\bigcup_{a \in \alpha^{\perp}} \Omega(a)$. Consequently, if we choose $c < C/c_0$, then \eqref{eqn.detailclaim} cannot happen, and hence such $c$ satisfies \eqref{eqn:mainthm}, proving the claim.

	Now fix $c > 0$ satisfying \eqref{eqn:mainthm}. 
	Since $\mathcal{R}$ is contained in a proper subcone of the generalized fibered cone, the set $\{ w \in \Omega : \alpha(w) = 0 \text{ for some } \alpha \in \mathcal{R}\}$ is finite. Denote by $w_1, \cdots, w_k$ its elements and set $R_i := \{ \alpha \in \Gamma^* \otimes \R : \alpha(w_i) = 0\}$ for $i = 1,\cdots, k$.
	As a result, we have finitely many hyperplanes $R_1, \cdots, R_k \subset \Gamma^* \otimes \R$ such that for any primitive integral element $\alpha \in \mathcal{R} - \bigcup_{i = 1}^k R_i$, we have $\alpha^{\perp} \cap \Omega = \{0\}$. Since the cohomology class corresponding to $\varphi$ is the dual of $(0, 1) \in \Gamma$, the hyperplanes $R_1, \cdots, R_k$ are disjoint from this cohomology class.
	This implies that for any $\alpha \in \mathcal{R} - \bigcup_{i = 1}^k R_i$ and non-trivial $a \in \alpha^{\perp}$, we have $(a \cdot D) \cap D = \emptyset$. Indeed, if $a = (x, m)$ in the coordinate of $\Gamma = H  \oplus \Z$, then $(a \cdot D) \cap D \neq \emptyset$ implies $$(x^{-1} \cdot D) \cap (\tilde \varphi^m \cdot D) \neq \emptyset.$$ Hence, $x^{-1} \in \Omega(\tilde \varphi^m)$, and therefore $$a = (x, m) \in - \Omega(\tilde \varphi^m) \times \{m\} \subset \Omega$$ which is a contradiction.

	Now let $\alpha \in \mathcal{R} - \bigcup_{i = 1}^k R_i$ with $n_{\alpha} > C$, and $h \in H$ be the element given by Lemma \ref{lem:geomarg}. We choose embedded spheres $S_1$ and $S_2$ in $D$ and $h \cdot D$ respectively. Since the translate of $D$ by any non-trivial element of $\alpha^\perp$ is disjoint from $D$ as observed in the previous paragraph, $S_1$ and $S_2$ are injected to the fiber $M_{\alpha} = \tilde M / \alpha^{\perp}$, and hence are vertices of the sphere complex $\mathcal{S}(M_{\alpha})$. Abusing notations, we also denote by $S_1$ and $S_2$ their images in $M_{\alpha}$. By \eqref{eqn.disjoint}, $S_1$ and $S_2$ are disjoint in $M_{\alpha}$, and therefore they represent adjacent vertices in $\mathcal{S}(M_{\alpha})$. Now it follows from  Proposition \ref{prop:gen_fibration} and \eqref{eqn:mainthm} that 
	$$\begin{aligned}
	d_{\mathcal{S}(M_{\alpha})}\left(S_2, \varphi_{\alpha}^{n_{\alpha} \cdot \left \lfloor c n_{\alpha}^{1/d} \right \rfloor }(S_2)\right) & \le d_{\mathcal{S}(M_{\alpha})}(S_2, S_1) + d_{\mathcal{S}(M_{\alpha})}\left(S_1, \varphi_{\alpha}^{n_{\alpha} \cdot  \left \lfloor c n_{\alpha}^{1/d} \right \rfloor }(S_2)\right) \\
	 & = 1 + d_{\mathcal{S}(M_{\alpha})}\left(S_1, {\varphi}^{\left \lfloor c n_{\alpha}^{1/d}\right \rfloor} (S_2)\right) = 2.
	\end{aligned}$$
	
	Therefore, we can estimate the asymptotic translation length of $\varphi_{\alpha}$ as follows: $$	l(\varphi_{\alpha}) \le \limsup_{m \to \infty} {d_{\mathcal{S}(M_{\alpha})}\left(S_2, \varphi_{\alpha}^{ n_{\alpha} \cdot  \left \lfloor c n_{\alpha}^{1/d} \right \rfloor \cdot m} (S_2)\right) \over n_{\alpha} \cdot \left \lfloor c n_{\alpha}^{1/d} \right \rfloor \cdot m} \le {2 \over n_{\alpha} \cdot \left \lfloor c n_{\alpha}^{1/d} \right \rfloor }.$$
	Since $n_{\alpha}$ is comparable to $\|\alpha\|$ for $\alpha \in \mathcal{R}$, this completes the proof of the estimate.
\end{proof}

\begin{proof}[Proof of Theorem \ref{theorem:mainthm}]
	By Theorem \ref{thm.weaker}, we have finitely many hyperplanes $R_1, \cdots, R_k$ so that the desired inequality \eqref{eqn.mainest} holds on $\mathcal{R} - \bigcup_{i = 1}^k R_i$. 
	
	Now for each $R_i$, choose a primitive integral element $\alpha_i \in R_i \cap \mathcal{R}$. By Proposition \ref{prop:gen_fibration}, this corresponds to a fibration $N \to S^1$ with the monodromy $\varphi_{\alpha_i}$ and the fiber $M_{\alpha_i} = \tilde M / \alpha_i^{\perp}$, with respect to the flow $\F$. We now apply Theorem \ref{thm.weaker} to $\varphi_{\alpha_i} : M_{\alpha_i} \to M_{\alpha_i}$. Since its mapping torus is also $N$ and the associated flow is $\F$ as well, the generalized fibered cone for $\varphi_{\alpha_i}$, which is equal to Fried's cone for $\F$ in $H^1(N)$ by Theorem \ref{thm.genfibfried}, is the same as the generalized fibered cone for the original diffeomorphism $\varphi : M \to M$.  In particular, $\mathcal{R}$ is still an intersection of a $d + 1$-dimensional rational subspace and a proper subcone of the generalized fibered cone for $\varphi_{\alpha_i}$. Hence, we can apply Theorem \ref{thm.weaker} to $\varphi_{\alpha_i} : M_{\alpha_i} \to M_{\alpha_i}$ and get finitely many hyperplanes $R_1^{(i)}, \cdots, R_{n_i}^{(i)}$ so that $\alpha_i \notin \bigcup_{j = 1}^{n_i} R_j^{(i)}$ and \eqref{eqn.mainest} holds on $\mathcal{R} - \bigcup_{j = 1}^{n_i} R_j^{(i)}$. Since $\alpha_i \in R_i$ while $\alpha_i$ is disjoint from $\bigcup_{j = 1}^{n_i} R_j^{(i)}$, the intersection 
	$R_i \cap R_j^{(i)}$ is a codimension 2 subspaces in $H^1(N)$.

	By the above argument, we now have finitely many codimension 2 subspaces $R_i \cap R_j^{(i)}$, $i = 1, \cdots, k$ and $j = 1, \cdots, {n_i}$, so that \eqref{eqn.mainest} holds on $\mathcal{R} - \bigcup_{i = 1}^{k} \bigcup_{j = 1}^{n_i} R_i \cap R_j^{(i)}$. By proceeding the above argument inductively, we finally obtain finitely many one-dimensional subspaces $L_1, \cdots, L_\ell$ in $H^1(N)$ such that \eqref{eqn.mainest} holds on $\mathcal{R} - \bigcup_{i = 1}^{\ell} L_{i}$. Since there are only finitely many primitive integral elements in $\bigcup_{i = 1}^{\ell} L_{i}$, this completes the proof.
\end{proof}


\section{Minimal translation lengths of Handlebody groups} \label{sec.handlebody}

Recall that we denote by $\H_g$ and $\D_g$ the handlebody group and the disk graph of a closed connected orientable surface of genus $g$ respectively.
In this section, we prove:

\mainthmhandle*

\begin{proof}
	We first deduce the lower bound from \cite{gadre2011minimal}. Recall that the inclusion $(D^2, \partial D^2) \subseteq (V_g, \partial V_g)$ induces a map from the vertices of $\D_g$ to the vertices of $\C_g$. Moreover, two non-isotopic disjoint disks in $V_g$ have non-isotopic disjoint boundary as noted in \cite[Lemma 2.3]{hensel2018primer}. This implies that the above map on vertices indeed gives the graph embedding $\mathcal{D}_g \to \mathcal{C}_g$. In other words, the disk graph $\mathcal{D}_g$ is a subgraph of the curve graph of $\partial V_g$, and therefore $$L_{\C_g}(\H_g) \le L_{\D_g}(\H_g)$$ for all $g \ge 2$. Since $L_{\C_g}(\Mod(S_g)) \asymp 1/g^2$ as proved by Gadre--Tsai \cite{gadre2011minimal} and $L_{\C_g}(\H_g) \ge L_{\C_g}(\Mod(S_g))$, the lower bound follows. 

	To prove the upper bound, we apply our main argument together with a similar construction as in \cite[Section 6]{hironaka2011fibered}. We begin with a genus 2 handlebody $V_2$ and a Torelli pseudo-Anosov $\varphi \in \H_2$. Such an element $\varphi$ indeed exists: consider a separating curve $b\subset \partial V_2$ which bounds a disk in $V_2$ as in Figure \ref{fig.handlebody2}. Taking any pseudo-Anosov $\varphi_0 \in \H_2$, two separating curves $b$ and $\varphi_0^n b$ fill $\partial V_2$ for large enough $n \in \N$ by \cite{MasurMinsky99}, and they bound disks in $V_2$. Therefore, $\varphi$ can be obtained by Thurston's construction \cite{thurston1988geometry} or Penner's construction \cite{penner1988construction} applied to the pair $b$ and $\varphi_0^n b$.

	\begin{figure}[h]
		\centering
		\begin{tikzpicture}[scale=1.2, every node/.style={scale=1}]
		\draw (-1.6, 0) .. controls (-1.6, 1) and (-0.6, 0.5) .. (0, 0.5) .. controls (0.6, 0.5) and (1.6, 1) .. (1.6, 0);
		\begin{scope}[rotate=180]
			\draw (-1.6, 0) .. controls (-1.6, 1) and (-0.6, 0.5) .. (0, 0.5) .. controls (0.6, 0.5) and (1.6, 1) .. (1.6, 0);
		\end{scope}
		
		\draw (-1.1, 0.1) .. controls (-1, -0.1) and (-0.6, -0.1) .. (-0.5, 0.1);
		\draw (-1, 0) .. controls (-0.9, 0.1) and (-0.7, 0.1) .. (-0.6, 0);
		
		\draw (1.1, 0.1) .. controls (1, -0.1) and (0.6, -0.1) .. (0.5, 0.1);
		\draw (1, 0) .. controls (0.9, 0.1) and (0.7, 0.1) .. (0.6, 0);
		
		\draw[red] (1, 0) .. controls (1.1, 0.25) and (1.5, 0.25) .. (1.6, 0);
		\draw[red, dashed] (1, 0) .. controls (1.1, -0.25) and (1.5, -0.25) .. (1.6, 0);
		
		\draw[red] (1.3, 0.3) node {$a$};
		
		\draw[blue] (0, 0.5) .. controls (0.2, 0.3) and (0.2, -0.3) .. (0, -0.5);
		\draw[blue, dashed] (0, 0.5) .. controls (-0.2, 0.3) and (-0.2, -0.3) .. (0, -0.5);
		
		\draw[blue] (0, -0.5) node[below] {$b$};
	\end{tikzpicture}
	\caption{A handlebody of genus 2} \label{fig.handlebody2}
\end{figure}

Now consider the homomorphism $H_1(\partial V_2) \to \Z$ given by the algebraic intersection number with $[a] \in H_1(\partial V_2)$ in Figure \ref{fig.handlebody2}. This induce $\Z$-covers $\tilde V \to V_2$ and $\partial \tilde V \to \partial V_2$ as in Figure \ref{fig.cyclic}. Since $\varphi \in \H_2$ is Torelli, it can be lifted to both covers. 

\begin{figure}[h]
	\centering
	\begin{tikzpicture}[scale=1.2, every node/.style={scale=1}]
	\begin{scope}[rotate=-90]
		\filldraw (0, 3.5) circle(1pt);
		\filldraw (0, 3.7) circle(1pt);
		\filldraw (0, 3.9) circle(1pt);

		\filldraw (0, -3.5) circle(1pt);
		\filldraw (0, -3.7) circle(1pt);
		\filldraw (0, -3.9) circle(1pt);
	
		\draw (-1.6, 0) .. controls (-1.6, 1) and (-0.6, 0.5) .. (0, 0.5) .. controls (0.2, 0.5) and (0.6, 0.7) .. (0.6, 1);
		\draw[red] (0.6, 1) .. controls (0.7, 1.25) and (1.1, 1.25) .. (1.2, 1);
		\draw[red, dashed] (0.6, 1) .. controls (0.7, 1-0.25) and (1.1, 1-0.25) .. (1.2, 1);

		\draw (1.2, 1) .. controls (1.2, 0.8) and (1, 0.5) .. (1, 0);

		\draw[thick, teal] (-1, 0) .. controls (-1.1, 0.25) and (-1.5, 0.25) .. (-1.6, 0);
		\draw[thick, dashed, teal] (-1, 0) .. controls (-1.1, -0.25) and (-1.5, -0.25) .. (-1.6, 0);
		
		\draw[teal] (-1.6, 0) node[above] {$z$};

		\begin{scope}[yscale=-1]

			\draw (-1.6, 0) .. controls (-1.6, 1) and (-0.6, 0.5) .. (0, 0.5) .. controls (0.2, 0.5) and (0.6, 0.7) .. (0.6, 1);
			\draw[red, dashed] (0.6, 1) .. controls (0.7, 1.25) and (1.1, 1.25) .. (1.2, 1);
			\draw[red] (0.6, 1) .. controls (0.7, 1-0.25) and (1.1, 1-0.25) .. (1.2, 1);

			\draw (1.2, 1) .. controls (1.2, 0.8) and (1, 0.5) .. (1, 0);
		\end{scope}
	
		\draw (-1.1, 0.1) .. controls (-1, -0.1) and (-0.6, -0.1) .. (-0.5, 0.1);
		\draw (-1, 0) .. controls (-0.9, 0.1) and (-0.7, 0.1) .. (-0.6, 0);
	

		\begin{scope}[shift={(0, 2)}]
			\draw (-1.6, 0) .. controls (-1.6, 1) and (-0.6, 0.5) .. (0, 0.5) .. controls (0.2, 0.5) and (0.6, 0.7) .. (0.6, 1);
			\draw[red] (0.6, 1) .. controls (0.7, 1.25) and (1.1, 1.25) .. (1.2, 1);
			\draw[red, dashed] (0.6, 1) .. controls (0.7, 1-0.25) and (1.1, 1-0.25) .. (1.2, 1);

			\draw (1.2, 1) .. controls (1.2, 0.8) and (1, 0.5) .. (1, 0);

			\begin{scope}[yscale=-1]

				\draw (-1.6, 0) .. controls (-1.6, 1) and (-0.6, 0.5) .. (0, 0.5) .. controls (0.2, 0.5) and (0.6, 0.7) .. (0.6, 1);

				\draw (1.2, 1) .. controls (1.2, 0.8) and (1, 0.5) .. (1, 0);
			\end{scope}
	
			\draw (-1.1, 0.1) .. controls (-1, -0.1) and (-0.6, -0.1) .. (-0.5, 0.1);
			\draw (-1, 0) .. controls (-0.9, 0.1) and (-0.7, 0.1) .. (-0.6, 0);
		\end{scope}


		\begin{scope}[shift={(0, -2)}]
			\draw (-1.6, 0) .. controls (-1.6, 1) and (-0.6, 0.5) .. (0, 0.5) .. controls (0.2, 0.5) and (0.6, 0.7) .. (0.6, 1);

			\draw (1.2, 1) .. controls (1.2, 0.8) and (1, 0.5) .. (1, 0);

			\begin{scope}[yscale=-1]

				\draw (-1.6, 0) .. controls (-1.6, 1) and (-0.6, 0.5) .. (0, 0.5) .. controls (0.2, 0.5) and (0.6, 0.7) .. (0.6, 1);
				\draw[red, dashed] (0.6, 1) .. controls (0.7, 1.25) and (1.1, 1.25) .. (1.2, 1);
				\draw[red] (0.6, 1) .. controls (0.7, 1-0.25) and (1.1, 1-0.25) .. (1.2, 1);

				\draw (1.2, 1) .. controls (1.2, 0.8) and (1, 0.5) .. (1, 0);
			\end{scope}
	
			\draw (-1.1, 0.1) .. controls (-1, -0.1) and (-0.6, -0.1) .. (-0.5, 0.1);
			\draw (-1, 0) .. controls (-0.9, 0.1) and (-0.7, 0.1) .. (-0.6, 0);
		\end{scope}

	\end{scope}
\end{tikzpicture}
\caption{Cyclic cover $\tilde V$} \label{fig.cyclic}
\end{figure}

Let $N$ be the mapping torus of $\partial V_2$ with the monodromy $\varphi$. Since $\varphi$ is pseudo-Anosov, $N$ is hyperbolic \cite{thurston1998hyperbolic}. Using the $\Z$-cover above, we can consider the generalized fibered cone $\C \subset H^1(N)$ which is two-dimensional. For each primitive integral $\alpha \in \C$, we know from Proposition \ref{prop:gen_fibration} that there is a fibration $N \to S^1$ with the monodromy $\varphi_{\alpha}$ and the fiber $S_{\alpha}$ which is the quotient of $\partial V$.
By \cite{hironaka2011fibered}, it follows that $\varphi_{\alpha}$ extends to the corresponding quotient of $\tilde V$ which is a handlebody.

Let $\mathcal{R} \subset \C$ be a two-dimensional proper subcone. For all but finitely many primitive integral element $\alpha \in \mathcal{R}$, as in the proof of Theorem \ref{thm.weaker}, we can choose an element $h$ in the Deck group and $k_{\alpha} \in \N$ comparable to $\|\alpha\|^2$ so that
the curve $z$ in Figure \ref{fig.cyclic} and its translate $h\cdot z$ inject into the fiber $S_{\alpha}$, and moreover $z$ has the image disjoint from images of both $h \cdot z$ and $\varphi_{\alpha}^{k_{\alpha}}(h \cdot z)$ in the quotient $S_{\alpha}$. We identify them with their images in $S_{\alpha}$.
Note that $z$, $h \cdot z$, and $\varphi_{\alpha}^{k_{\alpha}}(h\cdot z)$ bound disks in the handlebody. By \cite[Lemma 2.2]{hensel2018primer}, $z$ and $h\cdot z$ bound disjoint disks and similarly $z$ and $\varphi^{k_{\alpha}}_{\alpha}(h \cdot z)$ bound disjoint disks. Hence $h \cdot z$ and $\varphi_{\alpha}^{k_{\alpha}}(h \cdot z)$ represent vertices in the disk graph with distance at most 2. Since $\| \alpha \|$ is comparable to the genus $g_{\alpha}$ of $S_{\alpha}$, we conclude
$$l_{\D_{g_{\alpha}}}(\varphi_{\alpha}) \lesssim \frac{1}{g_\alpha^2}$$ for all primitive integral $\alpha \in \mathcal{R}$.
Since $N$ is hyperbolic, $\varphi_{\alpha}$ is pseudo-Anosov by \cite{thurston1998hyperbolic}. Since the monodromy $\varphi$ we started from is on the genus 2 surface, it is a simple computation using Thurston norm that all but finitely many natural numbers arise as $g_{\alpha}$ for primitive integral $\alpha \in \mathcal{R}$ (e.g. \cite[Proposition 6.10]{hironaka2011fibered}). This completes the proof.
\end{proof}


\section{Folding sequences of graph maps and (doubled) handlebodies}\label{sec:3mnfd}\

The results in this section are not completely new and follow from work of various authors. For instance, see \cite{laudenbach1974topologie}, \cite{hatcher2004homology}, \cite{luft1978actions}. 
Nevertheless, we provide a direct construction of the relevant diffeomorphisms for the sake of the readers. 

Let $G$ be a finite connected graph and $\psi: G\to G$ a homotopy equivalence. We assume that $f$ is \emph{combinatorial}, that is, $f$ maps vertices to vertices and edges to non-trivial edge-paths. The goal of this section is to provide one way of lifting a combinatorial homotopy equivalence of a graph to a diffeomorphism of a doubled handlebody and one of a handlebody. This liftings
are obviously not unique and may come from other kinds of constructions.

Denoting by $V(\cdot)$ the set of vertices, $\psi : G \to G$ gives a subdivision $G_{\Delta}$ of $G$ by setting $V(G_{\Delta})$ to be a union of $V(G)$ and the preimage of $V(G)$ under $\psi$. The graph $G_{\Delta}$ is topologically identical to $G$. Then $\psi$ is a composition $G \xrightarrow{i} G_{\Delta} \xrightarrow{\phi} G$ where $i : G \to G_{\Delta}$ is a subdivision map and $\phi : G_{\Delta} \to G$ is defined by $\phi(e) = \psi(i^{-1}(e))$ for an edge $e$ of $G_{\Delta}$. From the construction, $\phi$ is a graph map that sends an edge to an edge. By \cite{stallings1983topology}, there is a finite sequence of foldings (or, folding sequence) so that $\phi$ is a composition of those foldings. Here, folding on a graph is identifying two edges with a common endpoint. For details, see \cite{dowdall2015dynamics} and \cite{stallings1983topology}.

In this section, we explicitly construct a 3-manifold $M_G$ from $G$ and a diffeomorphism $\varphi : M_G \to M_G$ from a folding sequence of $\psi : G \to G$. More precisely, we fix a folding sequence $$G \xrightarrow{i} G_{\Delta} = G_0 \xrightarrow{\psi_0} G_1 \xrightarrow{\psi_1} \cdots \xrightarrow{\psi_n} G_{n+1} = G_{\Delta} \hookrightarrow G$$ where $\psi_i : G_i \to G_{i+1}$ is a folding, and then construct a diffeomorphism $\varphi_i : M_{G_i} \to M_{G_{i+1}}$ associated to $\psi_i$, with a canonical identification of $M_{G_{\Delta}}$ and $M_G$. As a result, composition of $\varphi_i$'s gives the desired diffeomorphism $\varphi : M_G \to M_G$. That is,
\begin{equation} \label{cd:folding}
	\begin{tikzcd}
	M_G \arrow[d, equal] \arrow[rr, dashed, "\varphi"] & & M_G \arrow[d, equal] \\
	M_{G_0} \arrow[r, "\varphi_0"] & \cdots \arrow[r, "\varphi_n"] & M_{G_{n+1}}
\end{tikzcd}
\end{equation}
Note here that a folding sequence for $\psi$ may not be unique; what we construct is a diffeomorphism $\varphi : M_G \to M_G$ respecting a fixed folding sequence for $\psi$. Throughout the paper, \emph{graph} is finite and connected.

\subsection{Doubled handlebody}
For a graph $G$, we construct a corresponding doubled handlebody $M_G$ as follows:
we first replace each vertex of $G$ with $S^3$ and each edge of $G$ with $S^2 \times I$, where $I$ is a compact interval. Then attachment of an edge to a vertex amounts to drilling out a 3-ball $D^3$ from $S^3$ and then gluing $S^2 \times I$ along one component of its boundary, as depicted in Figure \ref{fig:define3mnfd}.

\begin{figure}[h]
	\begin{tikzpicture}	[scale=0.5, every node/.style={scale=1}]
	\draw[fill=white!3, dashed] (0, 3) circle(1.5);
	\draw (-1, 4) node[left] {$S^3$};
	
	\draw[pattern = north east lines, pattern color = gray] (0, 0) circle(1);
	\draw[fill=white] (0, 0) circle(0.5);
	
	\draw (-0.5, 0) .. controls(-0.2, -0.2) and (0.2, -0.2) .. (0.5, 0);
	\draw[dashed] (-0.5, 0) .. controls (-0.2, 0.2) and (0.2, 0.2) .. (0.5, 0);	
	
	\draw (-1, 0) .. controls (-0.5, -0.5) and (0.5, -0.5) .. (1, 0);
	\draw[dashed] (-1, 0) .. controls (-0.5, 0.5) and (0.5, 0.5) .. (1, 0);
	
	\draw (-0.5, 1) node[left] {$S^2 \times I$};
	
	\end{tikzpicture}
	\qquad
	\begin{tikzpicture}[scale=0.5, every node/.style={scale=1}]
	
	\draw (0, 2.75);
	
	\draw[->, very thick] (-1, 0) -- (1, 0);
	
	\draw (0, -2.75);
	
	\end{tikzpicture}
	\qquad
	\begin{tikzpicture}[scale=0.5, every node/.style={scale=1}]
	
	\draw[fill=white!3, dashed] (0, 0) circle(2);
	
	\draw[pattern = north east lines, pattern color = gray] (0, 0) circle(1);
	\draw[fill=white] (0, 0) circle(0.5);
	
	\draw (-0.5, 0) .. controls(-0.2, -0.2) and (0.2, -0.2) .. (0.5, 0);
	\draw[dashed] (-0.5, 0) .. controls (-0.2, 0.2) and (0.2, 0.2) .. (0.5, 0);	
	
	\draw (-1, 0) .. controls (-0.5, -0.5) and (0.5, -0.5) .. (1, 0);
	\draw[dashed] (-1, 0) .. controls (-0.5, 0.5) and (0.5, 0.5) .. (1, 0);
	
	\draw (0, -2.75);
	\draw (0, 2.75);
	
	\end{tikzpicture}

	\caption{The gluing corresponding to an edge attached to a vertex.} \label{fig:define3mnfd}
\end{figure}
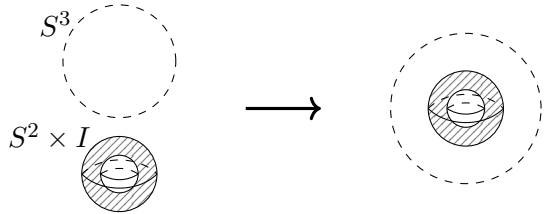 

Figure \ref{fig:twoedges} demonstrates two examples of induced 3-manifolds. Note that Figure \ref{fig:twoedgespart} will be used again in order to describe a folding map on $M_G$.
Furthermore, there is a map $\mathcal{P}: M_G\to G$ which sends the $S^3$ corresponding to each vertex to the vertex itself, and sends $S^2\times I$ to the corresponding edge by projection to the second component. It is easy to see that $\mathcal{P}$ induces an isomorphism between fundamental groups.

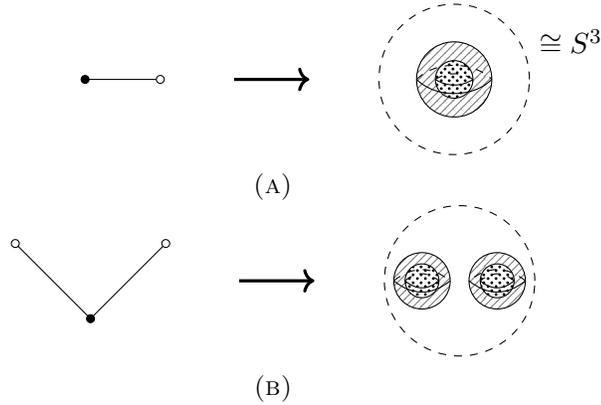
\begin{figure}[h]
	\centering
	\begin{subfigure}[b]{\textwidth}
		\centering
	\begin{tikzpicture}[scale=0.5, every node/.style={scale=1}]
	\draw (-1, 0) -- (1, 0);
	
	\filldraw (-1, 0) circle(3pt);
	\draw[fill=white] (1, 0) circle(3pt);
	
	\draw (0, -2);
	\draw (0, 2);  
	
	\draw (-5, 1);
	\end{tikzpicture}
	\qquad
	\begin{tikzpicture}[scale=0.5, every node/.style={scale=1}]
	\draw[->, very thick] (-1, 0) -- (1, 0);
	
	\draw (0, -2);
	\draw (0, 2);  
	\end{tikzpicture}
	\qquad
	\begin{tikzpicture}[scale=0.5, every node/.style={scale=1}]
	
	\draw[fill=white, dashed] (0, 0) circle(2);
	
	\draw[pattern = north east lines, pattern color = gray] (0, 0) circle(1);
	\draw[draw=white, fill=white] (0, 0) circle(0.5);
	\draw[pattern = crosshatch dots] (0, 0) circle(0.5);
	
	\draw (-0.5, 0) .. controls(-0.2, -0.2) and (0.2, -0.2) .. (0.5, 0);
	\draw[dashed] (-0.5, 0) .. controls (-0.2, 0.2) and (0.2, 0.2) .. (0.5, 0);	
	
	\draw (-1, 0) .. controls (-0.5, -0.5) and (0.5, -0.5) .. (1, 0);
	\draw[dashed] (-1, 0) .. controls (-0.5, 0.5) and (0.5, 0.5) .. (1, 0);
	
	\draw (2, 1) node[right] {$\cong S^3$};
	
	\end{tikzpicture}
	\caption{}
\end{subfigure}

	\begin{subfigure}[b]{\textwidth}
	\centering
	\begin{tikzpicture}[scale=0.5, every node/.style={scale=1}]
	
	\draw (-2, 2) -- (0, 0) -- (2, 2);
	
	\filldraw (0, 0) circle(3pt);
	\draw[fill=white] (-2, 2) circle (3pt);
	\draw[fill=white] (2, 2) circle (3pt);
	
	\draw (0, -1);
	
	\end{tikzpicture}
	\qquad
	\begin{tikzpicture}[scale=0.5, every node/.style={scale=1}]
	\draw (0, 2);
	\draw[->, very thick] (-1, 0) -- (1, 0);
	\draw (0, -2);
	\end{tikzpicture}
	\qquad
	\begin{tikzpicture}[scale=0.25, every node/.style={scale=1}]
	
	\draw[fill=white, dashed] (0, 0) circle(4);
	
	\draw[pattern=north east lines, pattern color = gray] (2, 0) circle(1.5);
	\draw[draw=white, fill=white] (2, 0) circle(0.9);
	\draw[pattern = crosshatch dots] (2, 0) circle(0.9);
	
	\draw (3.5, 0) .. controls (2.5, -0.8) and (1.5, -0.8) .. (0.5, 0);
	\draw[dashed] (3.5, 0) .. controls (2.5, 0.8) and (1.5, 0.8) .. (0.5, 0);
	
	\draw (2.9, 0) .. controls (2.3, -0.5) and (1.7, -0.5) .. (1.1, 0);
	\draw[dashed] (2.9, 0) .. controls (2.3, 0.5) and (1.7, 0.5) .. (1.1, 0);

	\draw[pattern = north east lines, pattern color = gray] (-2, 0) circle(1.5);
	\draw[draw=white, fill=white] (-2, 0) circle(0.9);
	\draw[pattern = crosshatch dots] (-2, 0) circle(0.9);
	
	\draw (-3.5, 0) .. controls (-2.5, -0.8) and (-1.5, -0.8) .. (-0.5, 0);
	\draw[dashed] (-3.5, 0) .. controls (-2.5, 0.8) and (-1.5, 0.8) .. (-0.5, 0);
	
	\draw (-2.9, 0) .. controls (-2.3, -0.5) and (-1.7, -0.5) .. (-1.1, 0);
	\draw[dashed] (-2.9, 0) .. controls (-2.3, 0.5) and (-1.7, 0.5) .. (-1.1, 0);
	
	\end{tikzpicture}
	\caption{} \label{fig:twoedgespart}
	\end{subfigure}
	
	\caption{3-manifolds obtained from the graphs. Light vertices in the graphs correspond to the dotted regions. Dark vertices in the graphs correspond to the outermost regions in the right figures. Edges of the graphs correspond the hatched regions diffeomorphic to $S^2 \times I$.} \label{fig:twoedges}
\end{figure}

For each of such foldings $\psi_i: G_i\to G_{i+1}$, let vertex $v$ of $G_i$ be the vertex on which the two edges to be folded are attached. Then the corresponding part in $M_{G_i}$ is described in Figure \ref{fig:twoedgespart}. Denote by $S$, $S_1$ and $S_2$ each part of the 3-manifold, induced from $v$ and other two vertices, respectively. In other words, $S$ is the outermost region in Figure \ref{fig:twoedgespart}, and $S_1$ and $S_2$ are dotted ones.

To get the folding on $M_{G_i}$, we extract some pieces as follows:\begin{enumerate}
	\item We drill out a solid cylinder $D^2 \times I \subseteq S^2 \times I$ from each of the two $S^2 \times I$ corresponding to an edge to be folded.
	\item Next, in $S_1$, we delete a small 3-ball $D^3$ whose boundary contains $S_1 \cap (D^2 \times \partial I)$ where $D^2 \times I$ is the cylinder removed in (1). Similarly, we drill out a small 3-ball in $S_2$.
	\item Finally we delete a cylinder $D^2 \times I$ in $S$ that connects two cylinders removed in (1). 
\end{enumerate} The union of deleted pieces is a 3-ball. As a result, we obtain $M_{G_i} \setminus D^3$ as in Figure \ref{fig:extractball}.

To ``fold" the manifold according to the folding of two edges in the graph, we make two corresponding $S^2 \times I$'s be contained in a single new $S^2 \times I$. Note that $(S^2 \times I)\setminus (D^2 \times I)$ has an annular face $\partial D^2 \times I$. Gluing two copies of them onto two opposite faces of $S^1 \times I \times I$, it results in $S^2 \times I$. In this regard, we glue $S^1 \times I \times I$ as indicated by patterns in Figure \ref{fig:solidtorus}. Then two copies of $(S^2 \times I) \setminus (D^2 \times I)$ corresponding to two edges get into a single $S^2 \times I$, representing the ``folding" of the manifold according to the folding of the edges.

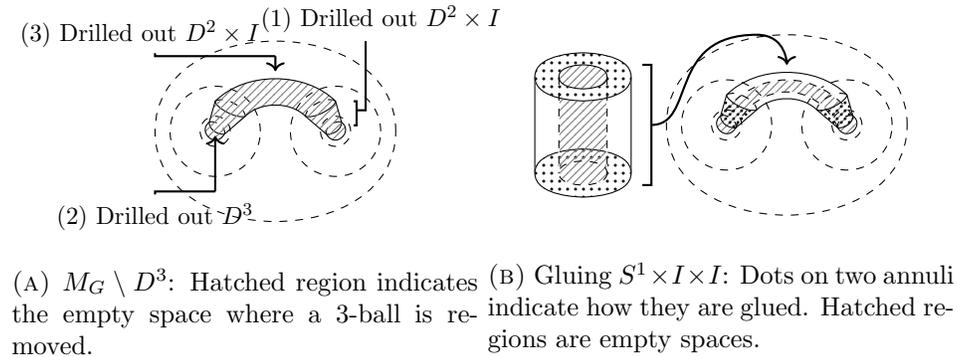
\begin{figure}[h]

\begin{subfigure}[H]{0.49\textwidth}
	\centering
	\begin{tikzpicture}[scale=0.4, every node/.style={scale=0.85}]
	
	\draw[fill=white, dashed] (4, 0) .. controls (4, 4) and (-4, 4) .. (-4, 0) .. controls (-4, -4) and (4, -4) .. (4, 0);
	
	\draw[fill=white!15, dashed] (2, 0) circle(1.5);
	\draw[fill=white!8, dashed] (2, 0) circle(0.5);
	
	\draw[fill=white!30, dashed] (2, 1) .. controls (1.7, 1.2) and (1, 0.8) .. (1, 0.5);
	\draw[fill=white!30] (1, 0.5) .. controls (1.3, 0.4) and (2, 0.8) .. (2, 1);
	
	\draw[fill=white!20, dashed] (2.2, 0.3) .. controls (2, 0.4) and (1.6,0.1) .. (1.7, -0.1);
	\draw[fill=white!20] (1.7, -0.1) .. controls (2, -0.1) and (2.2, 0.1) .. (2.2, 0.3);
	
	
	
	\draw[fill=white!15, dashed] (-2, 0) circle(1.5);
	\draw[fill=white!8, dashed] (-2, 0) circle(0.5);
	
	\draw[fill=white!30, dashed] (-2, 1) .. controls (-1.7, 1.2) and (-1, 0.8) .. (-1, 0.5);
	\draw[fill=white!30] (-1, 0.5) .. controls (-1.3, 0.4) and (-2, 0.8) .. (-2, 1);
	
	\draw[fill=white!20, dashed] (-2.2, 0.3) .. controls (-2, 0.4) and (-1.6,0.1) .. (-1.7, -0.1);
	\draw[fill=white!20] (-1.7, -0.1) .. controls (-2, -0.1) and (-2.2, 0.1) .. (-2.2, 0.3);
	
	
	
	
	\draw[pattern = north east lines, pattern color = gray] (-1.7, -0.1) .. controls (-2, -0.6) and (-2.6, 0) .. (-2.2, 0.3) -- (-2, 1) .. controls (-1, 2) and (1, 2) ..  (2, 1) -- (2.2, 0.3) .. controls (2.6, 0) and (2, -0.6) .. (1.7, -0.1) -- (1, 0.5) .. controls (0.5, 1) and (-0.5, 1) .. (-1, 0.5) -- (-1.7, -0.1);

	
	\draw[thick] (-4, 2.6) -- (-4, 2.5) -- (0, 2.5);
	\draw[thick, ->] (0, 2.5) -- (0, 2);
	
	\draw (-4.5, 2.5) node[above] {(3) Drilled out $D^2 \times I$};
	
	\draw[thick] (3, 3) -- (3, 0.7) -- (2.7, 0.7);
	\draw[thick] (2.6, 1) -- (2.7, 1) -- (2.7, 0.2) -- (2.6, 0.2);
	
	\draw (3.5, 3) node[above] {(1) Drilled out $D^2 \times I$};
	
	\draw[thick] (-4, -2.1) -- (-4, -2) -- (-2, -2);
	\draw[thick, ->] (-2, -2) -- (-2, -0.15);
	
	\draw (-4, -2.1) node[below] {(2) Drilled out $D^3$};
	
	\end{tikzpicture}

	\caption{$M_{G_i} \setminus D^3$: Hatched region indicates the empty space where a 3-ball is removed.} \label{fig:extractball}
\end{subfigure}
\begin{subfigure}[H]{0.49\textwidth}
	\centering
	\begin{tikzpicture}[scale=0.32, every node/.style={scale=1}]
	
	\draw (0, 5);
	\draw (0, -5);
	
	\draw[fill=white!20, draw=white] (-2, 2) .. controls (-2, 1) and (0, 1) .. (0, 1) .. controls (0, 1) and (2, 1) .. (2, 2) -- (2, -2) .. controls (2, -1) and (0, -1) .. (0, -1) .. controls (0, -1) and (-2, -1) .. (-2, -2) -- (-2, 2);

	\draw (2, 2) -- (2, -2);
	\draw (-2, 2) -- (-2, -2);

	\draw[draw = white, fill=white!15] (2, -2) .. controls (2, -3) and (0, -3) .. (0, -3) .. controls (0, -3) and (-2, -3) .. (-2, -2);
	\draw[draw = white, fill=white!15, dashed] (-2, -2) .. controls (-2, -1) and (0, -1) .. (0, -1) .. controls (0, -1) and (2, -1) .. (2, -2);
	\draw[pattern = dots] (2, -2) .. controls (2, -3) and (0, -3) .. (0, -3) .. controls (0, -3) and (-2, -3) .. (-2, -2).. controls (-2, -1) and (0, -1) .. (0, -1) .. controls (0, -1) and (2, -1) .. (2, -2);

	\draw[fill=white!10, draw = white] (1, -2) .. controls (1, -2.5) and (0, -2.5) .. (0, -2.5) .. controls (0, -2.5) and (-1, -2.5) .. (-1, -2) .. controls (-1, -1.5) and (0, -1.5) .. (0, -1.5) .. controls (0, -1.5) and (1, -1.5) .. (1, -2);

	\draw[pattern = north east lines, pattern color = gray, draw = white!20] (1, 2) .. controls (1, 2.5) and (0, 2.5) .. (0, 2.5) .. controls (0, 2.5) and (-1, 2.5) .. (-1, 2) -- (-1, -2)  .. controls (-1, -2.5) and (0, -2.5) .. (0, -2.5) .. controls (0, -2.5) and (1, -2.5) .. (1, -2) -- (1, 2);
	
	\draw[fill=white!20, draw=white] (2, 2) .. controls (2, 3) and (0, 3) .. (0, 3) .. controls (0, 3) and (-2, 3) .. (-2, 2) .. controls (-2, 1) and (0, 1) .. (0, 1) .. controls (0, 1) and (2, 1) .. (2, 2);
	\draw[pattern = dots] (2, 2) .. controls (2, 3) and (0, 3) .. (0, 3) .. controls (0, 3) and (-2, 3) .. (-2, 2) .. controls (-2, 1) and (0, 1) .. (0, 1) .. controls (0, 1) and (2, 1) .. (2, 2);
	
	\draw[fill=white!20, draw = white] (1, 2) .. controls (1, 2.5) and (0, 2.5) .. (0, 2.5) .. controls (0, 2.5) and (-1, 2.5) .. (-1, 2) .. controls (-1, 1.5) and (0, 1.5) .. (0, 1.5) .. controls (0, 1.5) and (1, 1.5) .. (1, 2);
	\draw[pattern = north east lines, pattern color = gray] (1, 2) .. controls (1, 2.5) and (0, 2.5) .. (0, 2.5) .. controls (0, 2.5) and (-1, 2.5) .. (-1, 2) .. controls (-1, 1.5) and (0, 1.5) .. (0, 1.5) .. controls (0, 1.5) and (1, 1.5) .. (1, 2);
	
		\draw[dashed] (1, -2) .. controls (1, -2.5) and (0, -2.5) .. (0, -2.5) .. controls (0, -2.5) and (-1, -2.5) .. (-1, -2) .. controls (-1, -1.5) and (0, -1.5) .. (0, -1.5) .. controls (0, -1.5) and (1, -1.5) .. (1, -2);

	\draw[dashed] (1, 2) -- (1, -2);
	\draw[dashed] (-1, 2) -- (-1, -2);
	\end{tikzpicture}
	\begin{tikzpicture}[scale=0.4, every node/.style={scale=1}]
	
	\draw[fill=white!3, dashed] (4, 0) .. controls (4, 4) and (-4, 4) .. (-4, 0) .. controls (-4, -4) and (4, -4) .. (4, 0);

	\draw[thick] (-4.8, 2) -- (-4.5, 2) -- (-4.5, -2) -- (-4.8, -2);
	
	\draw[thick, ->] (-4.5, 0) .. controls (-3, 0) and (-4, 3) .. (-1, 3) .. controls (0, 3) and (0, 2) .. (0, 2);

	\draw[fill=white!15, draw=white] (2, 0) circle(1.5);
	\draw[fill=white!8, draw = white] (2, 0) circle(0.5);
	
	\draw[fill=white!15, draw=white] (-2, 0) circle(1.5);
	
	\draw[fill=white!20, draw=white] (-1, 0.5) .. controls (-0.5, 1) and (0.5, 1) .. (1, 0.5) .. controls (1.3, 0.4) and (2, 0.8) .. (2, 1) .. controls (1, 2) and (-1, 2) .. (-2, 1) .. controls (-2, 0.8) and (-1.3, 0.4) .. (-1, 0.5);
	
	\draw[dashed] (2, 0) circle(1.5);

	\draw[fill=white!20, draw=white] (1.7, -0.1) .. controls (2, -0.1) and (2.2, 0.1) .. (2.2, 0.3) -- (2, 1) .. controls (2, 0.8) and (1.3, 0.4) .. (1, 0.5) -- (1.7, -0.1);
	
	\draw[dashed] (2, 0) circle(0.5);

	\draw[dashed] (-2, 0) circle(1.5);
	\draw[fill=white!8, draw=white] (-2, 0) circle(0.5);
	
	\draw[fill=white!20, draw=white] (-1.7, -0.1) .. controls (-2, -0.1) and (-2.2, 0.1) .. (-2.2, 0.3) -- (-2, 1) .. controls (-2, 0.8) and (-1.3, 0.4) .. (-1, 0.5) -- (-1.7, -0.1);
	
	\draw[draw = white, pattern = north east lines, pattern color = gray] (-1.7, -0.1) .. controls (-2, -0.6) and (-2.6, 0) .. (-2.2, 0.3) .. controls (-1, 1.9) and (1, 1.9) .. (2.2, 0.3) .. controls (2.6, 0) and (2, -0.6) .. (1.7, -0.1) .. controls (1, 1.5) and (-1, 1.5) .. (-1.7, -0.1);

	\draw[pattern = crosshatch dots, draw=white] (-1.7, -0.1) .. controls (-2, -0.1) and (-2.2, 0.1) .. (-2.2, 0.3) -- (-2, 1) .. controls (-2, 0.8) and (-1.3, 0.4) .. (-1, 0.5) -- (-1.7, -0.1);
	\draw[pattern = crosshatch dots, draw=white] (1.7, -0.1) .. controls (2, -0.1) and (2.2, 0.1) .. (2.2, 0.3) -- (2, 1) .. controls (2, 0.8) and (1.3, 0.4) .. (1, 0.5) -- (1.7, -0.1);

		\draw[dashed] (2, 1) .. controls (1.7, 1.2) and (1, 0.8) .. (1, 0.5);
	\draw (1, 0.5) .. controls (1.3, 0.4) and (2, 0.8) .. (2, 1);
	
	\draw[dashed] (2.2, 0.3) .. controls (2, 0.4) and (1.6,0.1) .. (1.7, -0.1);
	\draw (1.7, -0.1) .. controls (2, -0.1) and (2.2, 0.1) .. (2.2, 0.3);
	
		\draw (1, 0.5) -- (1.7, -0.1);
	\draw (2, 1) -- (2.2, 0.3);

	\draw[dashed] (-2, 0) circle(0.5);
	
	\draw[dashed] (-2, 1) .. controls (-1.7, 1.2) and (-1, 0.8) .. (-1, 0.5);
	\draw (-1, 0.5) .. controls (-1.3, 0.4) and (-2, 0.8) .. (-2, 1);
	
	\draw[dashed] (-2.2, 0.3) .. controls (-2, 0.4) and (-1.6,0.1) .. (-1.7, -0.1);
	\draw (-1.7, -0.1) .. controls (-2, -0.1) and (-2.2, 0.1) .. (-2.2, 0.3);
	
	\draw (-1, 0.5) -- (-1.7, -0.1);
	\draw (-2, 1) -- (-2.2, 0.3);

			\draw (1.7, -0.1) .. controls (2, -0.6) and (2.6, 0) .. (2.2, 0.3);
	
	\draw (-1.7, -0.1) .. controls (-2, -0.6) and (-2.6, 0) .. (-2.2, 0.3);
	
	\draw (-1, 0.5) .. controls (-0.5, 1) and (0.5, 1) .. (1, 0.5);
	\draw (-2, 1) .. controls (-1, 2) and (1, 2) ..  (2, 1);
	
	\draw[dashed] (2.2, 0.3) .. controls (1, 1.9) and (-1, 1.9) .. (-2.2, 0.3);
	\draw[dashed] (1.7, -0.1) .. controls (1, 1.5) and (-1, 1.5) .. (-1.7, -0.1);

	\end{tikzpicture}

	\caption{Gluing $S^1 \times I \times I$: Dots on two annuli indicate how they are glued. Hatched regions are empty spaces.} \label{fig:solidtorus}
\end{subfigure}
\caption{Folding of $M_{G_i}$}
\end{figure}

So far, we have seen how we ``fold" the manifold by gluing $S^2 \times I \times I$. After gluing as in Figure \ref{fig:solidtorus}, the remaining boundary is $S^2$: one annular face of $S^2 \times I \times I$ not glued and two 2-disks on the boundary of removed 3-balls in (3) of Figure \ref{fig:extractball}. See figure \ref{fig:remboundary}. 

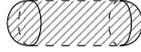
\begin{figure}[h]
	\begin{tikzpicture}[scale=0.3, every node/.style={scale=1}, rotate=90]

	\draw[fill=white!20, draw=white] (-1, 2) .. controls (-1, 1.5) and (0, 1.5) .. (0, 1.5) .. controls (0, 1.5) and (1, 1.5) .. (1, 2) -- (1, -2) .. controls (1, -1.5) and (0, -1.5) .. (0, -1.5) .. controls (0, -1.5) and (-1, -1.5) .. (-1, -2) -- (-1, 2);

	\draw[fill=white!8, draw=white] (1, 2) .. controls (1, 3) and (0, 3) .. (0, 3) .. controls (0, 3) and (-1, 3) .. (-1, 2) .. controls (-1, 1.5) and (0, 1.5) .. (0, 1.5) .. controls (0, 1.5) and (1, 1.5) .. (1, 2);

	\draw[fill=white!8, dashed] (1, -2) .. controls (1, -3) and (0, -3) .. (0, -3) .. controls (0, -3) and (-1, -3) .. (-1, -2) ;
	
	\draw[dashed] (1, 2) .. controls (1, 2.5) and (0, 2.5) .. (0, 2.5) .. controls (0, 2.5) and (-1, 2.5) .. (-1, 2);
	
	\draw[fill=white!10, draw=white] (1, -2) .. controls (1, -2.5) and (0, -2.5) .. (0, -2.5) .. controls (0, -2.5) and (-1, -2.5) .. (-1, -2) .. controls (-1, -1.5) and (0, -1.5) .. (0, -1.5) .. controls (0, -1.5) and (1, -1.5) .. (1, -2);
	
	\draw[dashed] (-1, -2) .. controls (-1, -1.5) and (0, -1.5) .. (0, -1.5) .. controls (0, -1.5) and (1, -1.5) .. (1, -2);
	
	\draw[fill=white!8, draw = white] (1, -2) .. controls (1, -3) and (0, -3) .. (0, -3) .. controls (0, -3) and (-1, -3) .. (-1, -2) .. controls (-1, -2.5) and (0, -2.5) .. (0, -2.5) .. controls (0, -2.5) and (1, -2.5) .. (1, -2);
	
	\draw[draw = white, pattern = north east lines, pattern color = gray] (1, 2) .. controls (1, 3) and (0, 3) .. (0, 3) .. controls (0, 3) and (-1, 3) .. (-1, 2) -- (-1, -2) .. controls (-1, -3) and (0, -3) .. (0, -3) .. controls (0, -3) and (1, -3) .. (1, -2) -- (1, 2);
	
	\draw(1, 2) .. controls (1, 3) and (0, 3) .. (0, 3) .. controls (0, 3) and (-1, 3) .. (-1, 2) .. controls (-1, 1.5) and (0, 1.5) .. (0, 1.5) .. controls (0, 1.5) and (1, 1.5) .. (1, 2);
	
	\draw (1, -2) .. controls (1, -3) and (0, -3) .. (0, -3) .. controls (0, -3) and (-1, -3) .. (-1, -2) .. controls (-1, -2.5) and (0, -2.5) .. (0, -2.5) .. controls (0, -2.5) and (1, -2.5) .. (1, -2);
	
	\draw[dashed] (1, 2) -- (1, -2);
	\draw[dashed] (-1, 2) -- (-1, -2);
	
	\end{tikzpicture}

	\caption{Empty (hatched) region in Figure \ref{fig:solidtorus}} \label{fig:remboundary}
\end{figure} 

Hence, we can glue a 3-ball along this boundary diffeomorphic to $S^2$. Gluing the 3-ball in this way represents the identifying endpoints of two folded edges (Figure \ref{fig:define3mnfd}). This whole procedure defines a map on $M_{G_i}$ corresponding to folding two edges on the graph. Moreover, adding the solid torus as in Figure \ref{fig:solidtorus} and then gluing a 3-ball along the boundary in Figure \ref{fig:remboundary} is just an adding a 3-ball to one in Figure \ref{fig:extractball}. Consequently, we get a diffeomorphism $\varphi_i : M_{G_i} \to M_{G_{i+1}}$ by folding two $S^2 \times I$ into one $S^2 \times I$.

Now as in \eqref{cd:folding}, we define $\varphi : M_G \to M_G$ to be a composition of maps $\varphi_i : M_{G_i} \to M_{G_{i+1}}$. Since each folding map $\varphi_{i} : M_{G_i} \to M_{G_{i+1}}$ is a diffeomorphism, $\varphi : M_G \to M_G$ is the desired diffeomorphism respecting the folding sequence of $\psi : G \to G$. 
As a consequence, we now have:

\begin{prop}\label{lem:dhand}
Let $G$ be a finite connected graph and $\psi: G\to G$ be a combinatorial homotopy equivalence. There is a doubled handlebody $M_G$, a map $\mathcal{P}: M_G\to G$ which induces an isomorphism on fundamental groups, and a diffeomorphism $\varphi: M_G \to M_G$ such that $\mathcal{P}\circ \varphi$ and $\psi\circ \mathcal{P}$ induces the same map on $\pi_1$.
\end{prop}

Note that the choice of diffeomorphism $\varphi$ is defined up to isotopy. In each of discussion in the next section, for a folding sequence of $\psi : G \to G$, we make a choice of specific \emph{diffeomorphism} $\varphi : M_G \to M_G$ constructed in the above way to consider its generalized fibered cone.

\subsection{Handlebody}

We can also carry out the construction above with ``half'' of every pieces.  More precisely, for a graph $G$, we get pieces by assigning each vertex to a 3-ball and each edge to a solid cylinder $D^2 \times I$ where $I = [0, 1]$. Then for each $D^2 \times I$, we attach $D^2 \times \{0\}$ to the ball assigned to the corresponding endpoint and similarly attach $D^2 \times \{1\}$ to the ball corresponding to the other endpoint. Now we get the handlebody $H_G$ with genus $\rk \pi_1(G)$. See Figure \ref{fig:handlebody}.

\begin{figure}[h]
	\centering
	\begin{tikzpicture}[scale=0.3, every node/.style={scale=1}]
	
	\draw(-14, 0) -- (-18, 0) -- (-22, 0);
	\filldraw (-14, 0) circle(6pt);
	\filldraw (-18, 0) circle(6pt);
	\filldraw (-22, 0) circle(6pt);
	
	\draw[fill=white] (-8, 0) circle(3);
	\draw[fill=white] (-2, 2) -- (-6, 2) .. controls (-7, 1) and (-7, -1) .. (-6, -2) -- (-2, -2);
	\draw[dashed] (-6, 2) .. controls (-5, 1) and (-5, -1) .. (-6, -2);
	
	\draw[fill=white] (0, 0) circle(3);
	\draw[fill=white] (6, 2) -- (2, 2) .. controls (1, 1) and (1, -1) .. (2, -2) -- (6, -2);
	\draw[dashed] (2, 2) .. controls (3, 1) and (3, -1) .. (2, -2);
	\draw[fill=white] (8, 0) circle(3);
	
	\draw (0, 0) node {$D^3$};
	\draw (0, 3) node[above] {$D^2 \times I$};
	\draw[->] (-2, 4) .. controls (-4, 4) .. (-4, 2.3);
	
	\end{tikzpicture}
	
	\caption{Attachment of 3-balls and solid cylinders according to a graph} \label{fig:handlebody}
\end{figure}
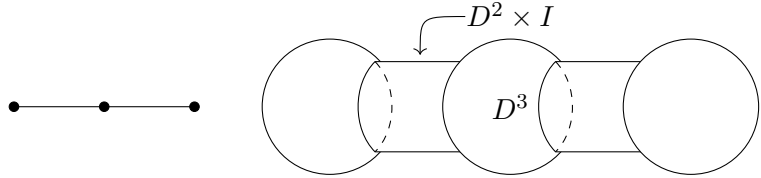

Now it remains to get a diffeomorphism corresponding to a folding on $G$ (or $G_{\Delta}$). To get the folding on the handlebody, we just proceed as in the previous subsection with half of the pieces. First, we remove $\mbox{(half-disk)}\times I$ from each $D^2 \times I$, assigned to an edge supposed to be folded. Here, we take the half-disk in $D^2$ so that the diameter of the half-disk is contained in $\partial D^2$. We similarly eliminate $\mbox{(half-disk)}\times I$ from the 3-ball corresponding to the common endpoint of the edges to be folded. We also remove a small half-ball from each of two balls corresponding to vertices supposed to be identified via the folding on $G$ (or $G_{\Delta}$). See Figure \ref{fig:foldingonhandlebody}.

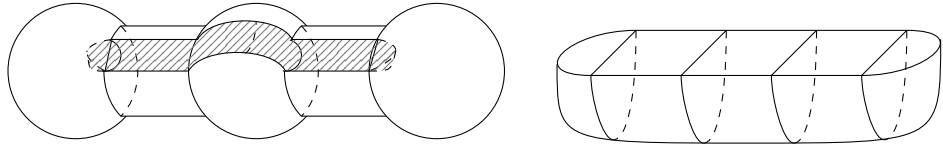
\begin{figure}[h]
	\centering
	\begin{subfigure}[h]{0.54\textwidth}
	\centering
	\begin{tikzpicture}[scale=0.3, every node/.style={scale=1}]
	
	\draw[fill=white] (-8, 0) circle(3);
	\draw[fill=white] (-2, 2) -- (-6, 2) .. controls (-7, 1) and (-7, -1) .. (-6, -2) -- (-2, -2);
	\draw[dashed] (-6, 2) .. controls (-5, 1) and (-5, -1) .. (-6, -2);
	
	\draw[fill=white] (0, 0) circle(3);
	\draw[fill=white] (6, 2) -- (2, 2) .. controls (1, 1) and (1, -1) .. (2, -2) -- (6, -2);
	\draw[dashed] (2, 2) .. controls (3, 1) and (3, -1) .. (2, -2);
	\draw[fill=white] (8, 0) circle(3);

	\draw[fill=white] (-6.7, 0.7) circle(20pt);
	\draw[fill=white] (1.3, 0.7) circle(20pt);
	
	\draw[fill=white, draw=white] (1.9-8, 2) .. controls (0.9-8, 1) and (0.9-8,-1) .. (1.9-8, -2) -- (0-8, -2) -- (0-8, 2) -- (1.9-8, 2);
	\draw[fill=white, draw=white] (1.9, 2) .. controls (0.9, 1) and (0.9,-1) .. (1.9, -2) -- (0, -2) -- (0, 2) -- (1.9, 2);
	
	\draw[dashed] (-0.5, 0.8) .. controls (-0.3, 0.8) and (0, 1.5) .. (0, 2.2);
	\draw[pattern = north east lines, pattern color = gray] (1.6-8, 1.4) -- (5.4-8, 1.4) .. controls (-2, 2.5) and (1, 2.5) .. (1.6, 1.4) -- (5.4, 1.4) -- (5, 0) -- (1.3, 0) .. controls (0.7 , 1.1) and (-2.4, 1.1) .. (5-8, 0) -- (1.3-8, 0) -- (1.6-8, 1.4);

	\draw[dashed, pattern = north east lines, pattern color = gray] (5.4, 1.4) .. controls (6.5, 1.2) and (6.5, 0.6) .. (5, 0);
	
	\draw[dashed, pattern = north east lines, pattern color = gray] (5, 0) .. controls (6, 0) .. (6.2, 1);
	
	\draw[dashed, pattern = north east lines, pattern color = gray] (1.6-8, 1.4) .. controls (1.6-8-1.4, 1.2) and (1.6-8-1.4, 0.6) .. (1.3-8, 0);
	
	\draw[dashed, pattern = north east lines, pattern color = gray] (1.3-8, 0) .. controls (1.3-8-0.7, 0) .. (1.6-8-1.1, 0.9);

	\end{tikzpicture}
	
	\caption{Hatched region is where the removed pieces were attached.}
\end{subfigure}
\begin{subfigure}{0.45\textwidth}
	\centering
	\begin{tikzpicture}[scale=0.3, every node/.style={scale=1}]
	
	\draw (-5, 1) -- (7, 1) .. controls (10, 1) and (8, -1) .. (5, -1) -- (-7, -1) .. controls (-10, -1) and (-8, 1) .. (-5, 1);
	\draw (7, 1) -- (5, -1);
	\draw (3, 1) -- (1, -1);
	\draw (-1, 1) -- (-3, -1);
	\draw (-5, 1) -- (-7, -1);
	
	\draw (8.5, 0.6) .. controls (8.5, -4) .. (0, -4) .. controls (-8.5, -4) .. (-8.5, -0.5);
	
	\draw (-7, -1) .. controls (-7, -2) and (-6.5, -3.85) .. (-6, -3.85);
	\draw[dashed] (-5, 1) .. controls (-5, -2.5) and (-5.5, -3.85)  .. (-6, -3.85);
	
	\draw (-7+4, -1) .. controls (-7+4, -2) and (-6.5+4, -4) .. (-6+4, -4);
	\draw[dashed] (-5+4, 1) .. controls (-5+4, -2.5) and (-5.5+4, -4)  .. (-6+4, -4);
		
	\draw (-7+4+4, -1) .. controls (-7+4+4, -2) and (-6.5+4+4, -4) .. (-6+4+4, -4);
	\draw[dashed] (-5+4+4, 1) .. controls (-5+4+4, -2.5) and (-5.5+4+4, -4)  .. (-6+4+4, -4);
		
	\draw (-7+4+4+4, -1) .. controls (-7+4+4+4, -2) and (-6.5+4+4+4, -3.85) .. (-6+4+4+4, -3.85);
	\draw[dashed] (-5+12, 1) .. controls (-5+12, -2.5) and (-5.5+12, -3.85)  .. (-6+12, -3.85);
	
	\end{tikzpicture}
	
	\caption{Removed $\mbox{(half-disk)} \times I$}
\end{subfigure}

	\caption{Drilling out from 3-balls and solid cylinders} \label{fig:foldingonhandlebody}
\end{figure}

Then gluing $\mbox{(half of }S^1\mbox{)} \times I \times I$ along some of its faces onto the removed regions, we finally get a handlebody removed a half-ball. See Figure \ref{fig:handlebodyremovedhalfball}.

\begin{figure}[H]
	\centering
	\begin{tikzpicture}[scale=0.3, every node/.style={scale=0.8}]
	
	\draw[pattern = north east lines, pattern color = gray!50] (8.5, 0.6) .. controls (8.5, -4) .. (0, -4) .. controls (-8.5, -4) .. (-8.5, -0.5);
	\draw[pattern = north east lines, pattern color = gray!50] (-5, 1) -- (7, 1) .. controls (10, 1) and (8, -1) .. (5, -1) -- (-7, -1) .. controls (-10, -1) and (-8, 1) .. (-5, 1);
	\draw (7, 1) -- (5, -1);
	\draw (3, 1) -- (1, -1);
	\draw (-1, 1) -- (-3, -1);
	\draw (-5, 1) -- (-7, -1);

	\draw (-7, -1) .. controls (-7, -2) and (-6.5, -3.85) .. (-6, -3.85);
	\draw[dashed] (-5, 1) .. controls (-5, -2.5) and (-5.5, -3.85)  .. (-6, -3.85);
	
	\draw[fill=white] (-0.5-4, -2.5) circle(0.5);
	\draw (-0.5-4, -2.5) node {$1$};
	
	\draw (-7+4, -1) .. controls (-7+4, -2) and (-6.5+4, -4) .. (-6+4, -4);
	\draw[dashed] (-5+4, 1) .. controls (-5+4, -2.5) and (-5.5+4, -4)  .. (-6+4, -4);
	
	\draw[fill=white] (-0.5, -2.5) circle(0.5);
	\draw (-0.5, -2.5) node {$2$};
	
	\draw (-7+4+4, -1) .. controls (-7+4+4, -2) and (-6.5+4+4, -4) .. (-6+4+4, -4);
	\draw[dashed] (-5+4+4, 1) .. controls (-5+4+4, -2.5) and (-5.5+4+4, -4)  .. (-6+4+4, -4);
	
	\draw[fill=white] (-0.5+4, -2.5) circle(0.5);
	\draw (-0.5+4, -2.5) node {$3$};
	
	\draw (-7+4+4+4, -1) .. controls (-7+4+4+4, -2) and (-6.5+4+4+4, -3.85) .. (-6+4+4+4, -3.85);
	\draw[dashed] (-5+12, 1) .. controls (-5+12, -2.5) and (-5.5+12, -3.85)  .. (-6+12, -3.85);
	

	\draw (-3+18, -1) .. controls (-3+18, -2) and (-2.5+18, -4) .. (-2+18, -4);
	\draw[dashed] (-1+18, 1) .. controls (-1+18, -2.5) and (-1.5+18, -4)  .. (-2+18, -4);
	
	\draw (1+18, -1) .. controls (1+18, -2) and (1.5+18, -4) .. (2+18, -4);
	\draw (3+18, 1) .. controls (3+18, -2.5) and (2.5+18, -4)  .. (2+18, -4);
	
	\draw (16, -4) -- (20, -4);
	
	\draw (19.5, -0.5) .. controls (19.5, -1.5) and (20.5, -1) .. (20.5, 0.5) -- (21, 1) -- (17, 1) -- (16.5, 0.5) .. controls  (16.5, -1) and (15.5, -1.5) .. (15.5, -0.5);
	
	\draw[fill=white] (19.5, -0.5) -- (15.5, -0.5) -- (15, -1) -- (19, -1) -- (19.5, -0.5);
	\draw (20.5, 0.5) -- (16.5, 0.5);
	
	\draw[fill=white] (17.5, -2.5) circle(0.5);
	\draw (17.5, -2.5) node {$2$};

	\draw[fill=white] (20, -2.4) circle(0.5);
	\draw (20, -2.4) node {$3$};
	
	\draw[fill=white] (14, 0) circle(0.5);
	\draw (14, 0) node {$1$};
	\draw[->] (14, -0.75) .. controls (14, -2.5) .. (15, -2.5);
	
	\draw (21, -1) node[right] {$= \mbox{(half of }S^1\mbox{)}\times I \times I$};
		
	\end{tikzpicture}
	
	\caption{Attaching $\mbox{(half of }S^1\mbox{)} \times I \times I$ to the removed region. The hatched region is empty due to the previous elimination, and the numbers indicate the way of gluing. Lines just stand for marking which part it comes from in Figure \ref{fig:foldingonhandlebody}, so there is no membrane or wall within the hatched region.} \label{fig:handlebodyremovedhalfball}
\end{figure}
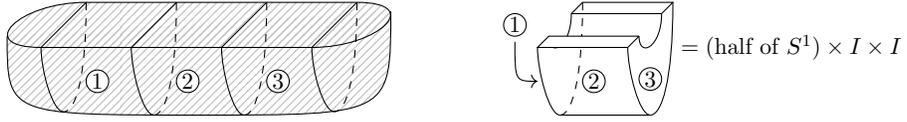

Gluing the half-ball as in Figure \ref{fig:regluinghalfball} again, we finally get the desired folding on the handlebody, and it is indeed a diffeomorphism.

\begin{figure}[h]
	\centering
	\begin{tikzpicture}[scale=0.3, every node/.style={scale=1}]
	
	\draw[pattern = north east lines, pattern color = gray!50] (8.5, 0.6) .. controls (8.5, -4) .. (0, -4) .. controls (-8.5, -4) .. (-8.5, -0.5);
	\draw[draw = white, pattern = north east lines, pattern color = gray!50] (-5, 1) -- (7, 1) .. controls (10, 1) and (8, -1) .. (5, -1) -- (-7, -1) .. controls (-10, -1) and (-8, 1) .. (-5, 1);
	
	\draw (7, 1) -- (5, -1);
	\draw (-5, 1) -- (-7, -1);
	
	\draw[draw=white, fill=white] (-6, -3.9) -- (-6, -4.1) -- (6, -4.1) -- (6, -3.9)-- (0, -3) -- (-6, -3.9);
	
	\draw[fill=white] (-7, -1) .. controls (-6, 0) and (4, 0) .. (5, -1) .. controls (-7+4+4+4, -2) and (-6.5+4+4+4, -3.85) .. (-6+4+4+4, -3.85) -- (-6, -3.85) .. controls (-6.5, -3.85) and (-7, -2) .. (-7, -1);
	
	\draw[fill=white] (-5, 1) .. controls (-4, 0) and (6, 0) .. (7, 1);
	
	\draw (-5, 1) -- (7, 1) .. controls (10, 1) and (8, -1) .. (5, -1) -- (-7, -1) .. controls (-10, -1) and (-8, 1) .. (-5, 1);

	\draw (-7, -1) .. controls (-7, -2) and (-6.5, -3.85) .. (-6, -3.85);
	\draw[dashed] (-5, 1) .. controls (-5, -2.5) and (-5.5, -3.85)  .. (-6, -3.85);
	
	\draw (-7+4+4+4, -1) .. controls (-7+4+4+4, -2) and (-6.5+4+4+4, -3.85) .. (-6+4+4+4, -3.85);
	\draw[dashed] (-5+12, 1) .. controls (-5+12, -2.5) and (-5.5+12, -3.85)  .. (-6+12, -3.85);
	
	\draw (-7+4, -1) .. controls (-7+4, -2) and (-6.5+4, -3.85) .. (-6+4, -3.85);
	\draw[dashed] (-5+4, 1) .. controls (-5+4, -2.5) and (-5.5+4, -3.85)  .. (-6+4, -3.85);

	\draw (-7+4+4, -1) .. controls (-7+4+4, -2) and (-6.5+4+4, -3.85) .. (-6+4+4, -3.85);
	\draw[dashed] (-5+4+4, 1) .. controls (-5+4+4, -2.5) and (-5.5+4+4, -3.85)  .. (-6+4+4, -3.85);

	\draw[dashed] (1.7, -0.3) .. controls (1.7, -2.3) and (2.3, -1.7) .. (2.3, 0.3);		\draw[dashed] (1.7-4, -0.3) .. controls (1.7-4, -2.3) and (2.3-4, -1.7) .. (2.3-4, 0.3);
	\draw[dashed] (-6, -3.85) .. controls (-3, -0.5) and (3, -0.5) .. (6, -3.85);
	
	\draw[pattern = north east lines, pattern color = gray!50] (17+3, -4) .. controls (16+3, -2) and (16+3, 0) .. (17+3, 2) arc(90:270:3);
	\draw (17+3, 2) .. controls (18+3, 0) and (18+3, -2) .. (17+3, -4);
	
	\draw[->] (17, 1.5) .. controls (17, 3.5) and (6, 3.5) .. (6, 1.5);

	\end{tikzpicture}
	
	\caption{Regluing the half-ball. Left one is a result of the gluing in Figure \ref{fig:handlebodyremovedhalfball}. The hatched regions are 2-dimensional faces that two pieces are glued.} \label{fig:regluinghalfball}
\end{figure}
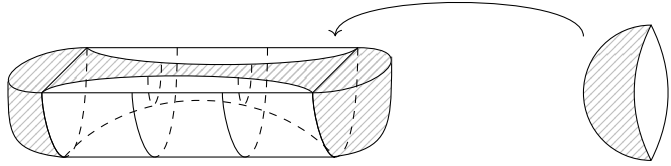

Hence, as a consequence, we have:

\begin{prop}\label{lem:hand}
Let $G$ be a finite connected graph and $\psi: G\to G$ be a combinatorial homotopy equivalence. There is a handlebody $H_G$, a map $\mathcal{P}: H_G\to G$ which induces an isomorphism on fundamental groups, and a diffeomorphism $\varphi: H_G \to H_G$ such that $\mathcal{P}\circ \varphi$ and $\psi\circ \mathcal{P}$ induces the same map on $\pi_1$.

\end{prop}


\section{Free-splitting complexes and free-factor complexes} \label{sec:estimate}

An immediate consequence of Theorem \ref{theorem:mainthm} and Proposition  \ref{lem:dhand} is the following:

\begin{cor} \label{thm:graph}
	Let $G$ be a finite connected graph and  $\psi: G \to G$ be a combinatorial homotopy equivalence. Let $M_G$ be the doubled handlebody and $\varphi: M_G\to M_G$ the induced diffeomorphism in Proposition \ref{lem:dhand}. Let $N$ be the mapping torus of $\varphi$. Let $\mathcal{R}$ be an intersection of a $d+1$-dimensional rational subspace and a proper subcone of the generalized fibered cone for $\varphi$. Then
	\[l(\varphi_{\alpha}) \lesssim \lVert\alpha\rVert^{-1-1/d}\]
	for all primitive integral element $\alpha \in \mathcal{R}$, where $\varphi_{\alpha}$ is the monodromy of the fibration $N \to S^1$ corresponding to $\alpha$, $l(\varphi_{\alpha})$ is the asymptotic translation length on the sphere complex of the fiber $M_{\alpha}$, and $\lVert\cdot\rVert$ be any norm on $H^1(N)$.
\end{cor}

\subsection{Positive cone} \label{subsec:positivecone}

We can obtain more results when we look at certain subcones of the generalized fibered cone. 

\begin{definition}[Train track map] \label{def:traintrackmap}
	Let $G$ be a finite connected graph, and consider a combinatorial homotopy equivalence $\psi:G \to G$.
	The map $\psi : G \to G$ is called a \emph{train track map} if for each edge $e$ and $n \ge 1$ the restriction $\psi^n|_e$ of $\psi^n$ to $e$ is an immersion, i.e. no back-tracking condition holds.
	\begin{itemize}
	
	\item A train track map is \emph{irreducible} if its transition matrix is irreducible.
	
	\item A train track map $\psi$ is said to be \emph{expanding} if the length of $\psi^n(e)$ diverges as $n \to \infty$ for each edge $e$.
	\end{itemize}
\end{definition}

\begin{remark}
	As above, some literature defines the train track map to be a homotopy equivalence. For instance, see \cite[Definition 2.11]{dowdall2015dynamics}. In contrast, the train track map has also been defined as a map that is not necessarily a homotopy equivalence. For example, see \cite[Section 2.1]{dowdall2017mcmullen}. 
\end{remark}

Let $G$ be a finite connected graph and $\psi : G \to G$ an expanding irreducible train track map. According to \cite{dowdall2015dynamics, dowdall2017mcmullen}, there is a proper subcone of a component of the symmetrized BNS-invariant $\Sigma_s$, called \emph{positive cone} and denoted by $\mathcal{A}$, containing the cohomology class corresponding to $\psi$. In the positive cone, each primitive integral class $\alpha$ corresponds to a fibration of the (folded) mapping torus of $\psi : G \to G$ over the circle whose monodromy map $\psi_{\alpha} :~G _{\alpha} \to G_{\alpha}$ is an expanding irreducible train track map.
 
By picking the fold in Section~\ref{sec:3mnfd} to correspond to the folds in the folded mapping torus of $\psi$, let $\varphi : M_G \to M_G$ be a diffeomorphism of a doubled handlebody $M_G$ constructed in Section \ref{sec:3mnfd}. Let $N$ be the mapping torus of $\varphi$ and $\mathcal{C}$ be the generalized fibered cone of $\varphi$.

\begin{thm} \label{thm:positivecone}
	
	Let $\psi: G \to G$ be an expanding irreducible train track map 
	and $\varphi : M_G \to M_G$ be an induced map on the doubled handlebody $M_G$ in Section \ref{sec:3mnfd}. Let $\mathcal{R}$ be an intersection of a $d+1$-dimensional rational subspace
	and a proper subcone of the generalized fibered cone of $\varphi$. Then 
	we have $$l(\varphi_{\alpha}) \lesssim g_{\alpha}^{-1-1/d}$$ for all primitive integral element $\alpha \in \mathcal{R}\cap \mathcal{A}$,  where $\varphi_{\alpha}$ is the monodromy of the fibration $N \to S^1$ corresponding to $\alpha$, $l(\varphi_{\alpha})$ is the asymptotic translation length on the sphere complex of the fiber $M_{\alpha}$, and  $g_{\alpha} = \operatorname{rank} \pi_1(G_\alpha)$.
\end{thm}

\begin{proof}

	From Corollary \ref{thm:graph}, we have already seen that \begin{equation} \label{eqn:mainestimate}
	l(\varphi_{\alpha}) \lesssim \lVert \alpha \rVert ^{-1-1/d}
	\end{equation} for a norm $\lVert \cdot \rVert$ on $H^1(N;\R)$, where $N$ is the mapping torus of $\varphi : M_G \to M_G$. Now it remains to see how $\lVert \alpha \rVert$ is related to $g_{\alpha}$. Since all norms on $H^1(N;\R)$, a finite-dimensional $\R$-vector space, are equivalent, we are free to choose the norm $\lVert \cdot \rVert$.
	
	In this line of thought, we introduce the Alexander norm on $H^1(N;\R)$ in a similar spirit of \cite{dowdall2015dynamics} to rewrite (\ref{eqn:mainestimate}) in terms of the genus of each fiber. Similar to the Thurston norm, the Alexander norm ball is the dual of Newton polytope of the Alexander polynomial $\Delta$. For details, see \cite{mcmullen2002alexander}.
	
	Denote $\lVert \alpha \rVert_A$ the Alexander norm of $\alpha$. As in \cite{dowdall2015dynamics}, it follows from \cite[Theorem 4.1]{mcmullen2002alexander} together with \cite[Theorem 3.1]{button2007mapping} that $$\lVert \alpha \rVert_A = g_{\alpha} - 1$$ when $\alpha$ belongs to the cone on the open faces of the Alexander norm ball. This equality is obtained in the following way (cf. \cite[Theorem 4.1]{mcmullen2002alexander}): Let $\alpha(\Delta)$ be the Laurant polynomial induced by $\alpha$ and $\Delta$. Writing $\Delta$ as the sum of distinct terms, there is only one summand which yields the highest degree term in $\alpha(\Delta)$ and similarly for the lowest degree term, since $\alpha$ is inside the cone. It means that $\deg \alpha(\Delta)$ is exactly the difference of the highest and the lowest degrees of induced terms from the summand of $\Delta$. On the other hand, the difference equals to $\lVert \alpha \rVert_A$, and thus $\deg \alpha(\Delta) = \lVert \alpha \rVert_A$. Combining with the fact that $g_{\alpha} = 1 + \deg \alpha(\Delta)$, we conclude the above equality.
	
	Even if $\alpha$ is not contained in the cone on an open face, $\lVert \alpha \rVert_A$ has something to do with $g_{\alpha}$. As one can see in the previous argument, the assumption of belonging to the cone is only for showing $\deg \alpha(\Delta) = \lVert \alpha \rVert_A$. Instead, if the assumption does not hold, then there can be two distinct summands of $\Delta$ deducing the highest (or the lowest) degree terms in $\alpha(\Delta)$ and thus cancellation may occur. As such, we obtain $\deg \alpha(\Delta) \le \lVert \alpha \rVert_A$ rather than the  equality. Then again from $g_{\alpha} = 1 + \deg \alpha(\Delta)$, we now conclude $$\lVert \alpha \rVert_A \ge g_{\alpha} - 1.$$
	
	Going back to the estimation (\ref{eqn:mainestimate}), we can now relate $\lVert \alpha \rVert$ (or $\lVert \alpha \rVert_{A}$) and $g_{\alpha}$ by $\lVert \alpha \rVert_A \ge g_{\alpha} - 1$, regardless of the position of $\alpha$ relative to the cones on the open faces of the Alexander norm ball. Consequently, we conclude that $$l(\varphi_{\alpha}) \lesssim g_{\alpha}^{-1-1/d}$$ as desired.
\end{proof}

Note here that the proper subcone $\mathcal{A}$ 
depends on the choice of a folding sequence of $\psi$. However, our argument does not depend on which folding sequence we choose. Indeed, the estimate in Theorem \ref{thm:positivecone} holds for any choice of a folding sequence.

\subsection{Free-splitting and free-factor complexes}
The \emph{free-splitting complex $\mathcal{FS}_K$} of a group $K$ is a simplicial complex consisting of free splittings of $K$. More precisely, its vertices are equivalence classes of free splittings of $K$ whose corresponding graph of groups have a single edge, and two vertices are connected by an edge of length 1 if they are represented by free splittings with a common refinement. For instance, two free splittings $A *(B * C)$ and $(A * B) * C$ are connected by an edge. For higher dimensional simplices and the equivalence relation among free splittings, see \cite{kapovich2014hyperbolicity}.

We continue the discussion from the previous subsection.
As  in \cite{aramayona2011automorphisms}, the sphere complex of the fiber $M_\alpha$ is equivalent to the free splitting complex of its fundamental group. Accordingly we can restate Corollary \ref{thm:graph} and Theorem \ref{thm:positivecone} in terms of free-splitting complexes as follows. We simply write $\mathcal{FS}_{g}$ the free-splitting complex of the free group $F_g$ of rank $g$:

\begin{cor} \label{cor:FSlength}
	
	Let $\psi: G \to G$ be an expanding irreducible train track map 
	and $\varphi : M_G \to M_G$ be the induced map on the doubled handlebody $M_G$ in Proposition \ref{lem:dhand}. 
	Let $N$ be the mapping torus of $\varphi$. Let $\mathcal{R}$ be an intersection of a $d+1$-dimensional rational subspace and a proper subcone of the generalized fibered cone of $\varphi$. Then
	\[l_{\mathcal{FS}_{\pi_1(M_{\alpha})}}(\varphi_{\alpha}) \lesssim \lVert\alpha\rVert^{-1-1/d}\]
	for all primitive integral element $\alpha \in \mathcal{R}$, where $\varphi_{\alpha}$ is the monodromy of the fibration $N \to S^1$ corresponding to $\alpha$ with the fiber $M_{\alpha}$ and $\lVert\cdot\rVert$ be any norm on $H^1(N)$.

	Moreover,\[
	l_{\mathcal{FS}_{g_{\alpha}}}(\varphi_{\alpha})\lesssim g_{\alpha}^{-1-1/d}
	\] for all primitive integral element $\alpha \in \mathcal{R} \cap \mathcal{A}$ where $g_{\alpha} = \operatorname{rank} \pi_1(G_\alpha)$.
	
\end{cor}

Similar to the free-splitting complex, the \emph{free-factor complex $\mathcal{FF}_g$} of $F_g$ is a simplicial complex whose vertices are conjugacy classes of proper free factors of $F_g$, and $k+1$ vertices form a $k$-simplex if they can be represented by proper free factors $A_0 \le A_1 \le \cdots \le A_{k}$ of $F_g$. Again, we set all edges of the free-factor complex to be of length 1. 

As in \cite{kapovich2014hyperbolicity}, there is a coarsely $\Out(F_g)$-equivariant (coarseness independent of $g$) Lipschitz map $\Psi$ from the vertices of the free-splitting complex to the vertices of the free-factor complex. Hence, the following  analogous result for the free-factor complex is deduced from Corollary \ref{cor:FSlength}:

\begin{cor} \label{cor:FFlength}
	Let $\psi: G \to G$ be an expanding irreducible train track map 
	and $\varphi : M_G \to M_G$ be the induced map on the doubled handlebody $M_G$ in Proposition \ref{lem:dhand}. 
	Let $N$ be the mapping torus of $\varphi$. Let $\mathcal{R}$ be an intersection of a $d+1$-dimensional rational subspace and a proper subcone of the generalized fibered cone of $\varphi$. Then \[
	l_{\mathcal{FF}_{g_{\alpha}}}(\varphi_{\alpha}) \lesssim g_{\alpha}^{-1-1/d}
	\] for all primitive integral element $\alpha \in \mathcal{R} \cap \mathcal{A}$ where $\varphi_{\alpha}$ is the monodromy of the fibration $N \to S^1$ corresponding to $\alpha$ with the fiber $M_{\alpha}$ and $g_{\alpha} = \operatorname{rank} \pi_1(G_\alpha)$.
\end{cor}

For more relations among complexes defined on a free group, one can refer to \cite{guirardel2019boundaries} and \cite{kapovich2009geometric}.

\begin{remark} \label{rmk.out}
	It was studied by Hironaka \cite{hironaka2011fibered} and Hensel--Kielak \cite{HenselKielak_handlebody} that when monodromies in a given Thurston's fibered cone can be extended to associated handlebodies. Indeed, in \cite{HenselKielak_handlebody}, they proved that for any free group automorphism $f : F_g \to F_g$, there are infinitely many pseudo-Anosov $\varphi \in \H_g$ such that every monodromy in Thurston's fibered cone containing $\varphi$ extends to the associated handlebody. Also, Laudenbach \cite{laudenbach1974topologie} showed that for a doubled  handlebody $\mathcal{M}_g$ of genus $g$, the natural surjection $\Mod(\mathcal{M}_g) \to \Out(F_g)$ has the kernel acting trivially on the sphere complex of $\mathcal{M}_g$. Together with these results, Corollary \ref{cor:FSlength} and Corollary \ref{cor:FFlength} can be interpreted in terms of asymptotic translation lengths of $\Out(F_g)$-elements.
\end{remark}

\medskip
\bibliographystyle{alpha} 
\bibliography{spherical}

\end{document}